\documentclass[10pt]{amsart}
\usepackage{amsmath,amssymb,amscd,amsthm,amsfonts,amstext,amsbsy,mathrsfs,hyperref,upgreek,mathtools,stmaryrd,enumitem,bbm}
\usepackage{turnstile}

\usepackage[textsize=footnotesize,color=yellow, bordercolor=white]{todonotes}

\usepackage
[a4paper, margin=1.5in]{geometry}

\hypersetup{colorlinks=true}

\newcommand{\gr}{\mathrm{gr}}

\newcommand{\Col}{\operatorname{Col}}

\newcommand{\Ord}{{\mathrm{Ord}}}

\newcommand{\ZFC}{{\sf ZFC}}

\newcommand{\forces}[2]{\Vdash^{#1}_{#2}}
\newcommand{\BS}{\omega^\omega}
\newcommand{\lpm}{\mathcal{L}_{\mathrm{pm}}} 
\newcommand{\PI}{\boldsymbol\Pi}
\newcommand{\SIGMA}{\boldsymbol\Sigma}
\newcommand{\DELTA}{\boldsymbol\Delta}

\newtheorem{theorem}{Theorem}[section]
\newtheorem{lemma}[theorem]{Lemma}

\newtheorem{proposition}[theorem]{Proposition}

\newtheorem{question}[theorem]{Question}

\newcounter{cl}[theorem]
\newtheorem{claim}[cl]{Claim}
\newtheorem*{claim*}{Claim}
\newtheorem*{subclaim*}{Subclaim}

\theoremstyle{definition}
\newtheorem{definition}[theorem]{Definition}

\newtheorem{example}[theorem]{Example}
\theoremstyle{remark}
\newtheorem{remark}[theorem]{Remark}

\newtheorem{case}{Case}
\newtheorem*{case*}{Case}

\newenvironment{enumerate-(a)}{\begin{enumerate}[label={\upshape (\alph*)}, leftmargin=2pc]}{\end{enumerate}}

\newenvironment{enumerate-(a)-r}{\begin{enumerate}[label={\upshape (\alph*)}, leftmargin=2pc,resume]}{\end{enumerate}}

\newenvironment{enumerate-(A)}{\begin{enumerate}[label={\upshape (\Alph*)}, leftmargin=2pc]}{\end{enumerate}}

\newenvironment{enumerate-(A)-r}{\begin{enumerate}[label={\upshape (\Alph*)}, leftmargin=2pc,resume]}{\end{enumerate}}

\newenvironment{enumerate-(i)}{\begin{enumerate}[label={\upshape (\roman*)}, leftmargin=2pc]}{\end{enumerate}}

\newenvironment{enumerate-(i)-r}{\begin{enumerate}[label={\upshape (\roman*)}, leftmargin=2pc,resume]}{\end{enumerate}}

\newenvironment{enumerate-(I)}{\begin{enumerate}[label={\upshape (\Roman*)}, leftmargin=2pc]}{\end{enumerate}}

\newenvironment{enumerate-(I)-r}{\begin{enumerate}[label={\upshape (\Roman*)}, leftmargin=2pc,resume]}{\end{enumerate}}

\newenvironment{enumerate-(1)}{\begin{enumerate}[label={\upshape (\arabic*)}, leftmargin=2pc]}{\end{enumerate}}

\newenvironment{enumerate-(1)-r}{\begin{enumerate}[label={\upshape (\arabic*)}, leftmargin=2pc,resume]}{\end{enumerate}}

\begin{document}




\author{Juan P. Aguilera}
\address{Juan P. Aguilera, Department of Mathematics, Ghent University. Krijgslaan 281-S8, 9000 Ghent, Belgium.} 
\address{Institut f\"ur diskrete Mathematik und
  Geometrie, Technische Universit\"at Wien. Wiedner Hauptstrasse 8-10,
  1040 Wien, Austria.} 
\email{aguilera@logic.at}

\author{Sandra M\"uller} 
\address{Sandra M\"uller, Institut f\"ur Mathematik, Universit\"at Wien. Kolingasse 14-16, 1090
  Wien, Austria.} 
\email{mueller.sandra@univie.ac.at} 

\author{Philipp Schlicht} 
\address{Philipp Schlicht, Institut f\"ur Mathematik, Universit\"at Wien. Kolingasse 14-16, 1090 Wien, Austria} 
\address{School of Mathematics, University of Bristol, Fry Building, Woodland Road, Bristol, BS8 1UG, UK} 
\email{philipp.schlicht@univie.ac.at}

\date{\today}

\title{Long Games and $\sigma$-Projective Sets}

\begin{abstract} 
We prove a number of results on the determinacy of $\sigma$-projective sets of reals, i.e., those belonging to the smallest pointclass containing the open sets and closed under complements, countable unions, and projections. We first prove the equivalence between $\sigma$-projective determinacy and the determinacy of certain classes of games of variable length ${<}\omega^2$ (Theorem 2.4). We then give an elementary proof of the determinacy of $\sigma$-projective sets from optimal large-cardinal hypotheses (Theorem 4.4). Finally, we show how to generalize the proof to obtain proofs of the determinacy of $\sigma$-projective games of a given countable length and of games with payoff in the smallest $\sigma$-algebra containing the projective sets, from corresponding assumptions (Theorems 5.1 and 5.4).
\end{abstract} 

\maketitle

\section{Introduction}
Let $\BS$ denote the space of infinite sequences of natural numbers
with the product topology, i.e., the topology generated by basic
(cl)open sets of the form
\[ O(s) = \{x\in\BS: \text{ $x$ extends $s$}\}, \] where
$s \in \omega^{<\omega}$. As usual, we will refer to the elements of
$\BS$ as \emph{reals}. Given a subset $A$ of $\BS$, the
\emph{payoff set}, we consider the Gale-Stewart game $G$ of length $\omega$
as follows:

\begin{figure}[h]
\centering
\begin{tabular}{c|cccccc}
 $\mathrm{I}$ & $x_0$ & & $x_2$ & & $\dots$ &\\ \hline
 $\mathrm{II}$ & & $x_1$ & & $x_3$ & & $\dots$
\end{tabular}
\; \; for $x_0, x_1, \hdots \in \omega$.
\end{figure}

Two players, I and II, alternate turns playing
$x_0, x_1,\hdots \in \omega$ to produce an element
$x = (x_0, x_1, \dots)$ of $\BS$. Player I wins if and only if
$x\in A$; otherwise, Player II wins. One can likewise define longer games by
considering subsets of $\omega^\alpha$, where $\alpha$ is a countable
ordinal. If so, we will again regard $\omega^\alpha$ as a product of
discrete spaces. A game is \emph{determined} if one of players I and II has a winning strategy. A set $A\subset\omega^\omega$ is said to be determined if the corresponding game is.

 These games have been studied extensively; under suitable set-theoretic assumptions, one can prove various classes of them to be determined.
One often studies the determinacy of pointclasses given in terms of definability (a general reference is Moschovakis \cite{Mo09}). A pointclass central to this article is the following:

\begin{definition}
The pointclass of \emph{$\sigma$-projective} sets is the smallest pointclass closed under complements, countable unions,\footnote{Of course, one only considers unions of sets in the same space, as is usual.} and projections.
\end{definition}
In this article, we consider the following classes of games, and their interplay:
\begin{enumerate}
\item games of fixed countable length $\alpha$ whose payoff is $\sigma$-projective;
\item games of variable length ${<}\alpha+\omega^2$ whose payoff is a pointclass containing the clopen sets and contained in the $\sigma$-projective sets;
\item games of countable length $\alpha$ with payoff in other $\sigma$-algebras.
\end{enumerate}

Neeman extensively studied long games and their
connection to large cardinals in \cite{Ne04}. These results are based
on his earlier work in \cite{Ne95} and \cite{Ne02} on games of length
$\omega$, where he started connecting moves in long games with moves
in iteration games. He showed in \cite{Ne04} for example that the
determinacy of games with fixed countable length and analytic payoff
set follows from the existence of Woodin cardinals. Moreover, he also analyzed games of
\emph{continuously coded} length (see also \cite{Ne05}) and games of
length up to a \emph{locally uncountable} ordinal. In \cite{Ne07} he
even showed determinacy for open games of length $\omega_1$, indeed
for a larger class of games of length $\omega_1$, from large
cardinals. That the determinacy of arbitrary games of length
$\omega_1$ is inconsistent is due to Mycielski and has been known for
a long time (see \cite{My64}). 

As for the converse, the determinacy of infinite games implies the existence of inner models with large cardinals (cf. e.g., \cite{Ha78,KW10,Tr13,Uh16,MSW}  and others).\\

\paragraph{{\bf Summary of results}} We begin in Section \ref{SectSimple} by introducing a class of games of variable length below $\omega^2$. We call these games \emph{$\Gamma$-simple}, where $\Gamma$ is a pointclass. We show that $\sigma$-projective determinacy implies the determinacy of $\Gamma$-simple games of length $\omega^2$ where $\Gamma$ is the pointclass of all $\sigma$-projective sets. In Section \ref{sec:sigmaproj} we introduce a class of games we call \emph{decoding games}. These are used to show that $\sigma$-projective determinacy follows from simple clopen determinacy of length $\omega^2$. The proof also shows directly that simple $\sigma$-projective determinacy of length $\omega^2$ follows from simple clopen determinacy of length $\omega^2$

In Section \ref{sec:simpleclopen}, we prove simple clopen determinacy of length $\omega^2$ (and thus $\sigma$-projective determinacy of length $\omega$) from optimal large cardinal assumptions. The proof is level-by-level. An alternative, purely inner-model-theoretic proof of $\sigma$-projective determinacy can be found \cite{Ag18}; our proof, however, requires little inner model theory beyond the definition of the large-cardinal assumption.
Roughly, it consists in repeatedly
applying a theorem of Neeman \cite{Ne04} to reduce a simple
clopen game of length $\omega^2$ to an iteration game on
an extender model with many partial extenders. The difference here is
that players are allowed to drop gratuitously in the iteration game
finitely many times to take advantage of the partial extenders in the
model. 

Finally, in Section \ref{sec:concl}, we exhibit some additional applications of the proof in Section \ref{sec:simpleclopen}. Specifically we prove from (likely optimal) large cardinal assumptions that $\sigma$-projective games of length $\omega\cdot\theta$ are determined, where $\theta$ is a countable ordinal. We also prove, from a hypothesis slightly beyond projective determinacy, that games in the smallest $\sigma$-algebra containing the projective sets are determined.

\section{Simple games of length $\omega^2$}\label{SectSimple}
We begin by noting a result on the determinacy of games of length\footnote{Here and below, we identify
  the spaces $\omega^{\omega^2}$, $(\BS)^\omega$ and
  $\omega^{\omega\times\omega}$, as well as $\omega^{\omega\cdot n}$
  and $(\omega^\omega)^n$.} ${<}\omega^2$:
\begin{theorem}[folklore] \label{TheoremPDFolk}
The following are equivalent:
\begin{enumerate}
\item All projective games of length $\omega$ are determined.
\item All projective games of length $\omega\cdot n$ are determined, for all $n<\omega$.
\item All clopen\footnote{In the product topology on $\omega^{\omega\cdot n}$.} games of length $\omega\cdot n$ are determined, for all $n<\omega$.
\end{enumerate}
\end{theorem}
Instead of providing a proof of Theorem \ref{TheoremPDFolk}, we refer the reader to the proof of Theorem \ref{main} below, an easy adaptation of which suffices (cf. Remark \ref{RemarkLocal}).

Our first result is the analog of Theorem \ref{TheoremPDFolk} for $\sigma$-projective games. Although the equivalence between the first two items remains true if one replaces ``projective'' with ``$\sigma$-projective,'' the one between the last two does not.
Instead, one needs to consider a larger class of games that are still decided in less than $\omega^2$ rounds, in the sense
that for any $x\in \omega^{\omega^2}$ there is $n\in\omega$ such
that for all $y\in \omega^{\omega^2}$, if
$y\upharpoonright \omega\cdot n = x \upharpoonright\omega\cdot n$,
then
\[x \text{ is a winning run for Player I if and only if $y$ is.}\]
Note that every game of
length $\omega\cdot n$ can be seen as a game of this form, for any
$n\in\omega$. For the definition below, we adapt the convention
that if $n\in\omega$, then a subset $A$ of
$\omega^{\omega\cdot n}$ can be identified with
\[\{x\in\omega^{\omega^2}: x\upharpoonright \omega\cdot n \in A\}.\]

\begin{definition}\label{DefSimple}
  Let $\Gamma$ be a collection of subsets of $\omega^{\omega^2}$
  (each $A \in \Gamma$ identified with a subset of $\omega^{\omega\cdot n}$ for some
  $n \in \omega$ as above). A game of length $\omega^2$ is
  $\Gamma$-\emph{simple} if it is obtained as follows:
\begin{enumerate}
\item For every $n \in \omega$, games that are decided after $\omega \cdot n$ moves such that their payoff restricted to sequences of length $\omega \cdot n$ is in $\Gamma$ are $\Gamma$-simple.
\item \label{eq:DefSimpleI} Let $n\in\omega$ and for each
  $i \in \omega$ let $G_i$ be a $\Gamma$-simple game. Then the game
  $G$ obtained as follows is $\Gamma$-simple: Players I and II take
  turns playing natural numbers for $\omega \cdot n$ moves, i.e., $n$
  rounds in games of length $\omega$. Afterwards, Player I plays some
  $i\in\omega$. Players I and II continue playing according to the
  rules of $G_i$ (keeping the first $\omega \cdot n$ natural numbers
  they have already played, but not $i$).
\item \label{eq:DefSimpleII} Let $n\in\omega$ and for each
  $i \in \omega$ let $G_i$ be a $\Gamma$-simple game. Then the game
  $G$ obtained as follows is $\Gamma$-simple: Players I and II take
  turns playing natural numbers for $\omega \cdot n$ moves, i.e., $n$
  rounds in games of length $\omega$. Afterwards, Player II plays some
  $i\in\omega$. Players I and II continue playing according to the
  rules of $G_i$ (keeping the first $\omega \cdot n$ natural numbers
  they have already played, but not $i$).
\end{enumerate}
\end{definition}

We are mainly interested in games of length $\omega^2$ which are \emph{simple
  clopen}, i.e., which are $\Gamma$-simple for $\Gamma = \DELTA^0_1$
the collection of clopen sets.
Let us start by noting that every simple clopen game has in fact a
payoff set which is clopen in $\omega^{\omega\times\omega}$ (this fact will not be needed, but it justifies our terminology).

\begin{lemma}\label{LemmaSimpleClopenIsClopen}
  Let $G$ be a simple clopen game of length $\omega^2$ with payoff set
  $B$. Then $B$ is clopen in $\omega^{\omega\times\omega}$.
\end{lemma}
\begin{proof}
  We prove this by induction on the definition of simple clopen
  games. In the case that $G$ is a clopen game of fixed length
  $\omega \cdot n$ it is clear that $B$ is clopen. So suppose that we
  are given simple clopen games $G_i$, $i<\omega$ with payoff sets
  $B_i$. Moreover, suppose for notational simplicity that $n = 1$,
  i.e., the players play one round of length $\omega$ before Player I
  plays $i \in \omega$ to decide which rules to follow. Inductively,
  we can assume that every $B_i$ is clopen in
  $\omega^{\omega\times\omega}$. Let $G$ be the game obtained by
  applying \eqref{eq:DefSimpleI} in Definition \ref{DefSimple}.
  For $x \in \omega^{\omega\times\omega}$, write $x^*$ for
  \[ x\upharpoonright \omega ^\frown x\upharpoonright [\omega+1,
    \omega^2). \] Then $x$ is a winning run for Player I in $G$ iff
\begin{align*}
  x \in \bigcup_{i\in\omega}\left\{
  y\in \omega^{\omega\times\omega}:  y(\omega) = i \wedge y^* \in B_i
  \right\}.
\end{align*}
Each $B_i$ is open, so the payoff set $B$ of $G$ is open. Additionally, $x \in B$ if,
and only if,
\begin{align*}
  x \in \bigcap_{i\in\omega}\left\{
  y\in \omega^{\omega\times\omega}:  y(\omega) \neq i \vee y^* \in B_i
  \right\}.
\end{align*}
Each $B_i$ is closed, so $B$ is closed. Therefore, $B$ is clopen.

The argument for applying \eqref{eq:DefSimpleII} in Definition
\ref{DefSimple} is analogous.
\end{proof}

The main fact about simple clopen games is that their determinacy 
is already equivalent to
determinacy of $\Gamma$-simple games where $\Gamma$ is the pointclass
of all $\sigma$-projective sets:

\begin{theorem}\label{main}
The following are equivalent:
\begin{enumerate}
\item All $\sigma$-projective games of length $\omega$ are determined.
\item All simple $\sigma$-projective games of length $\omega^2$ are determined.
\item All simple clopen games of length $\omega^2$ are determined.
\end{enumerate}
\end{theorem}

The proof of $\sigma$-projective determinacy of length $\omega$ from simple clopen determinacy, i.e., $(3) \Rightarrow (1)$, will take place in the next section (Proposition \ref{PropositionSC}). 
For now, we content ourselves with showing that $\sigma$-projective determinacy of length $\omega$ implies simple clopen determinacy of length $\omega^2$, i.e., $(1) \Rightarrow (3)$ (see Proposition \ref{PropositionSimpleClopenEasyProof}), although the proof we give easily adapts to show simple $\sigma$-projective determinacy of length $\omega^2$ from the same hypothesis, i.e., $(1) \Rightarrow (2)$ (see Proposition \ref{PropositionSC2}). Note that $(2) \Rightarrow (3)$ is obvious.
It will be convenient for the future to introduce the definition of the \emph{game
  rank} of a simple clopen game. 

To each simple clopen game $G$ of length
$\omega^2$ we associate a countable ordinal $\gr(G)$, the game
  rank of $G$, by induction on the definition of simple clopen
games. If $G$ is a game of fixed length $\omega\cdot n$, then
$\gr(G) = n$. If $G$ is obtained from games $G_0, G_1, \hdots$, and
from an ordinal $\omega \cdot n$ as in Definition \ref{DefSimple}, we
let
\[ \gr(G) = \sup\{\gr(G_i)+\omega:i\in\omega\} + n. \] 

\begin{remark} \label{RemarkLocal}
The proof of Theorem \ref{main} is local. We leave the computation of the precise complexity bounds to the curious reader, but we mention that e.g., the proof shows the equivalence among 
\begin{enumerate}
\item the determinacy of games of length $\omega$ which are $\PI^1_n$ for some $n\in\omega$;
\item the determinacy of simple clopen games of length $\omega^2$ of rank ${<}\omega$
\item the determinacy of simple games of rank ${<}\omega$ which are $\PI^1_n$ for some $n\in\omega$.
\end{enumerate}
Thus, Theorem \ref{main} generalizes Theorem \ref{TheoremPDFolk}.
\end{remark}

The reason why we have chosen to define the game rank this way is that it will make some arguments by induction easier later on.
Let us consider some examples: if
$\gr(G) = \omega$, then $G$ is essentially a game in which an infinite
collection of games, each of bounded length, are given, and one of the
players begins by deciding which one of them they will play. If
$\gr(G) = \omega+1$, then the game is similar, except that the player
does not decide which game they will play until the first $\omega$
moves have been played. More generally, a game has limit rank if and only if it begins with one player choosing one among a countably infinite collection of games that can be played. 

\begin{remark}
Let $G$ be a simple clopen game of rank $\alpha$ and $p$ be a partial play of $G$. Denote by $G_p$ the game that results from $G$ after $p$ has been played. Then $G_p$ is a simple clopen game of rank $\leq \alpha$.
\end{remark}

Clearly, every simple clopen game has a countable rank. Now we turn to the proof of $(1) \Rightarrow (3)$ in Theorem \ref{main}.

\begin{proposition}\label{PropositionSimpleClopenEasyProof}
Suppose that all $\sigma$-projective games of length $\omega$ are determined. Then all simple clopen games of length $\omega^2$ are determined.
\end{proposition}
\begin{proof}
Let us say that two games $G$ and $H$ are \emph{equivalent} if the following hold:
\begin{enumerate}
\item Player I has a winning strategy in $G$ if and only if she has one in $H$; and
\item Player II has a winning strategy in $G$ if and only if she has one in $H$.
\end{enumerate}

We prove the proposition by induction on the game rank of a simple clopen game $G$. 
In the case that the game rank is a successor ordinal, we additionally show that the game is equivalent to a $\sigma$-projective game of length $\omega$ (this is clear in the limit case). Suppose that $\alpha$ is a limit ordinal and this has been shown for games of rank ${<}\alpha$. Let $\alpha+n$ be the rank of $G$, where $n\in\omega$. If $n = 0$, then the result follows easily: by the definition of game rank, the rules of $G$ dictate that one player must begin by choosing one amongst an infinite sequence of games $G^i$. If that player has a winning strategy in any one of them, then choosing that game will guarantee a win in $G$; otherwise, the induction hypothesis yields a winning strategy for the other player in each game $G^i$ and thus in $G$. 

If $0<n$, say, $n = k+1$, then one argues as follows. Given a partial play $p$ of $G$, we denote by $G_p$ the game $G$ after $p$ has been played. 
Consider the following game, $H$:
\begin{enumerate}
\item Players I and II alternate $\omega$ many turns to produce a real number $x$.
\item Afterwards, the game ends. Player I wins if and only if 
\[\text{$\exists x_1\in\omega^\omega\, \forall y_1\in\omega^\omega\, \exists x_2\in\omega^\omega\,\hdots\, \forall y_k\in\omega^\omega$ Player I has a winning strategy in $G_p$,}\]
where $p = \langle x, x_1*y_1, \hdots, x_k*y_k \rangle$.\footnote{Here, $x*y$ denotes the result of facing off the strategies coded by the reals $x$ and $y$.}
\end{enumerate}

\begin{claim}\label{cl:1}
$H$ is equivalent to $G$. 
\end{claim}

Granted the claim, it is easy to prove the proposition, for, letting $p$ be as above, the rules of $G_p$ dictate that one of the players must choose one amongst an infinite sequence of games $G^i$. Assume without loss of generality that this is Player I.
By induction hypothesis, the set of $p$ such that there exists $i$ so that Player I has a winning strategy in $G^i$ is $\sigma$-projective. Hence, the payoff set of $H$ is $\sigma$-projective, as was to be shown. 

It remains to prove the claim. Thus, let $0 \leq m \leq k$ and consider the following game $H_m$:
\begin{enumerate}
\item Players I and II alternate $\omega\cdot (m+1)$ many turns to produce real numbers $z_0,\hdots, z_m$.
\item Afterwards, the game ends. Player I wins if and only if 
\[\text{$\exists x_{m+1}\, \forall y_{m+1}\, \exists x_{m+2}\,\hdots\, \forall y_{k}$ Player I has a winning strategy in the game $G_p$,}\]
where $p = \langle z_0,\hdots, z_m, x_{m+1}*y_{m+1}, \hdots, x_{k}*y_{k} \rangle$.
\end{enumerate}
Thus, $H = H_0$.
By induction on $q = k-m$ (i.e., by downward induction on $m$), we show that $H_m$ is equivalent to $G$. A simple modification of the argument shows that $(H_m)_p$ is equivalent to $G_p$, for each $p \in (\omega^\omega)^{l}$ and each $l\leq m+1$; this is possible because $G_p$ is a simple clopen game of rank $\leq \alpha+n$.
We will use the equivalence between $(H_m)_p$ and $G_p$ as part of the induction hypothesis. 

We have shown (by the induction hypothesis for the proposition) that $G$ is equivalent to $H_k$. The same argument applied to $G_p$ shows that $G_p$ is equivalent to $(H_k)_p$ for every $p \in (\BS)^{k+1}$ and thus that $G_p$ is determined. Moreover, Player I wins a run $p \in (\BS)^{m+1}$ of $H_m$ if and only if she has a winning strategy in $(H_{m+1})_p$ and Player II wins a run $p \in (\BS)^{m+1}$ of $H_m$ if and only if Player I does not have a winning strategy in $(H_{m+1})_p$; however, $(H_{m+1})_p$ is determined for every $p \in (\BS)^{m+1}$, as it is a $\sigma$-projective game of length $\omega$ (this follows from an argument as right after the statement of Claim \ref{cl:1}). Moreover, the induction hypothesis (for the claim) shows that a player has a winning strategy in $(H_{m+1})_p$ if and only if she has one in $G_p$. This shows that a player has a winning strategy in $H_m$ if and only if she has one in $G$, as was to be shown.
\end{proof}

\section{Decoding games}\label{sec:sigmaproj}

Our first goal in this section is to prove the following proposition, i.e., $(3) \Rightarrow (1)$ in Theorem \ref{main}.

\begin{proposition}\label{PropositionSC}
  Suppose simple clopen games of length $\omega^2$ are
  determined. Then $\sigma$-projective games of length $\omega$ are
  determined.
\end{proposition}

In order to prove Proposition \ref{PropositionSC}, we introduce a
representation of $\sigma$-projective sets.

\begin{definition}
Fix an enumeration 
\[\{A_{n+1}:n\in\omega\}\]
of all basic open and basic closed sets in each $(\BS)^k$, for
$1\leq k<\omega$.  Suppose $A\subset (\BS)^k$ is $\sigma$-projective.  A
\emph{$\sigma$-projective code} $[A]$ of $A$ is defined inductively as
follows:
\begin{enumerate}
\item If $A$ is basic open, then $[A] = \langle n \rangle$, where $A = A_n$.
\item If $A$ is basic closed, then $[A] = \langle n \rangle$, where $A = A_n$.
\item If $A = \bigcup_i A_i$, then $[A] = \langle [A_0], [A_1], \hdots \rangle$.
\item If $A = \bigcap_i A_i$, then $[A] = \langle 0, [A_0], [A_1], \hdots \rangle$.
\item If $A = (\BS)^k \setminus B$ for some $k \in \omega$, then
  $[A] = \langle 1, [B] \rangle$.
\item If $A = p[B]$, then $[A] = \langle 2, [B] \rangle$.
\item If $A = u[B]$, then $[A] = \langle 3, [B] \rangle$.
\end{enumerate}
\end{definition}
Here, $u[B]$ denotes the dual of the projection,
 \[ u[B] = \{x\in\BS:\forall y\, (x,y) \in B\}. \]

A $\sigma$-projective code of a given set is not unique. In fact,
\begin{lemma}
Let $A$ be $\sigma$-projective. 
\begin{enumerate}
\item $A$ has a $\sigma$-projective code in which no complements and
  no basic open sets appear.
\item $A$ has a $\sigma$-projective code in which no complements and
  no basic closed sets appear.
\end{enumerate}
\end{lemma}
\begin{proof}
  Given a $\sigma$-projective code for $A$, one obtains, by a simple
  application of de Morgan's laws, a $\sigma$-projective code in which
  complements are only applied to basic open sets or to basic closed
  sets. 
  Since every basic open set is clopen, basic open sets can be
  replaced by a countable intersection of basic closed sets in the
  projective code; similarly, basic closed sets can be replaced by a
  countable union of basic open sets.
\end{proof}

\begin{definition}\label{def:decgame}
  Suppose $A\subset (\BS)^n$ for some $1\leq n < \omega$ is
  $\sigma$-projective and fix a code $[A]$ for $A$. The \emph{decoding
    game for $A$} (with respect to $[A]$) is a game of length
  $\omega^2$ according to the following rules: In the first $n$
  rounds, which we will call the \emph{preparation}, Players I and II
  start by alternating turns playing $\omega\cdot n$ natural numbers
  to obtain reals $x_1, x_2, \hdots, x_n \in \BS$. Afterwards, they
  proceed according to the $\sigma$-projective code $[A]$ of $A$ via
  the following recursive definition:
\begin{enumerate}
\item If $A$ is basic open, say $A = A_k$, then the game is over,
  i.e., further moves are not relevant. Player I wins if and only if
  $(x_1, x_2, \hdots, x_n) \in A_k$.
\item If $A$ is basic closed, say $A = A_k$, then the game is over,
  i.e., further moves are not relevant. Player I wins if and only if
  $(x_1, x_2, \hdots, x_n) \in A_k$.
\item If $A = \bigcup_i A_i$, so that
  $[A] = \langle [A_0], [A_1], \hdots \rangle$, then Player I plays
  some $k \in \omega$. The game continues from the current play (without the last move, $k$)
  with the rules of the decoding game with respect to $[A_k]$.
\item If $A = \bigcap_i A_i$, so that
  $[A] = \langle 0, [A_0], [A_1], \hdots \rangle$, then Player II
  plays some $k \in \omega$. The game continues from the current
  play (without the last move, $k$) with the rules of the decoding game with respect to $[A_k]$.
\item If $A = (\BS)^k\setminus B$ for some $k \in \omega$, so that
  $[A] = \langle 1, [B] \rangle$, then the game continues from the
  current play with the rules of the decoding game with respect to
  $[B]$, except that the roles of Players I and II are reversed.
  \label{defeq:decgamecomp}
\item If $A = p[B]$, so that $[A] = \langle 2, [B] \rangle$, then
  Player I plays some $y \in \BS$ in $\omega$ moves of the game, where
  the moves of Player II are not relevant. The game continues from the
  current play, together with $y$, using the rules of the decoding
  game with respect to $[B]$. \label{defeq:decgamep}
\item If $A = u[B]$, so that $[A] = \langle 3, [B] \rangle$, then
  Player II plays some $y \in \BS$ in $\omega$ moves if the game,
  where the moves of Player I are not relevant. The game continues
  from the current play, together with $y$, using the rules of the
  decoding game with respect to $[B]$. \label{defeq:decgameu}
\end{enumerate}
\end{definition}
Clause \eqref{defeq:decgamecomp} of the preceding definition will not be used below, but we have defined it in the natural way nonetheless.

\begin{lemma}\label{LemmaSCimpliesSP0}
  Let $A\subset(\BS)^n$ for some $1 \leq n < \omega$ be
  $\sigma$-projective and fix a code $[A]$ for $A$. A player has a
  winning strategy in the game with payoff set $A$ if and only if
  the player has a winning strategy in the decoding game for $A$ with
  respect to $[A]$.
\end{lemma}
\begin{proof}
  Assume first that Player I has a winning strategy $\sigma$ in the
  game $G_A$ with winning set $A$. Then the following describes a
  winning strategy for Player I in the decoding game for $A$ with
  respect to $[A]$. In the first $n$ rounds of the game, Player I
  follows the strategy $\sigma$. Since $\sigma$ is a winning strategy
  in $G_A$, the players produce a sequence of reals
  $(x_1, \dots, x_n) \in A$. In the following rounds, Player I follows
  the rules of the decoding game according to $[A]$, playing witnesses
  to the fact that $(x_1, \dots, x_n) \in A$. This means that, e.g., 
  if $[A] = \langle [A_0], [A_1], \hdots \rangle$, i.e.,
  $A = \bigcup_i A_i$, then Player I plays $k \in \omega$ such that
  $(x_1, \dots, x_n) \in A_k$. The strategy for the other cases is
  defined similarly. This yields a winning strategy for Player I in
  the decoding game for $A$ with respect to $[A]$.

  Now assume that Player I has a payoff strategy $\sigma$ in the
  decoding game for $A$ with respect to $[A]$. Then the restriction of
  $\sigma$ to the first $n$ rounds of the game is a winning strategy
  for Player I in the game $G_A$. Similarly for Player II.
\end{proof}

\begin{lemma} \label{LemmaSCimpliesSP}
  Let $A\subset(\BS)^k$ be $\sigma$-projective, for some $1\leq k <\omega$. Then, for every
  $\sigma$-projective code $[A]$ for $A$ in which no complements
  appear, the decoding game given by $[A]$ is simple clopen.
\end{lemma}
\begin{proof}
  Choose a code $[A]$ for $A$ in which no complements appear. We will
  show by induction on the definition of simple clopen games, that the
  game obtained as in \eqref{defeq:decgamep} or \eqref{defeq:decgameu}
  in Definition \ref{def:decgame} is simple clopen again. 
  This will
  finish the proof as games obtained by clauses (1)-(4) in Definition \ref{def:decgame} are
  clearly simple clopen, by the definition of simple clopen games.
  
  Let us introduce some notation for this proof. If $x\in\omega^\omega$, we write $x^I$ for the sequence of digits of $x$ in even positions, and $x^{II}$ for the sequence of digits of $x$ in odd positions. Thus, if $x$ results from a run of a Gale-Stewart game, then $x^I$ is the sequence of moves of Player I and $x^{II}$ that of Player II.

  So let $G$ be a simple clopen game with payoff set
  $B \subseteq (\BS)^\omega$ and consider the game $G^{P,k}$ which is
  defined as follows. Players I and II start by alternating turns
  playing $\omega \cdot k$ natural numbers to define reals
  $x_1, \dots, x_k \in \BS$. Then players I and II alternate moves to play some real $x_{k+1}$. Finally, Players I and II alternate again to produce reals
  $z_1, z_2, \dots$ so that
  $(x_1, \dots, x_k, x_{k+1}^{I}, z_1, z_2, \dots) \in (\BS)^\omega$ and we say
  that $(x_1, \dots, x_k, x_{k+1}, z_1, z_2, \dots)$ is a winning run
  for Player I in $G^{P,k}$ iff
  $(x_1, \dots, x_k, x_{k+1}^{\mathrm{I}}, z_1, z_2, \dots) \in B$. That means if $B^*$
  denotes the payoff set of $G^{P,k}$, we have
  $(x_1, \dots, x_k, x_{k+1}, z_1, z_2, \dots) \in B^*$ iff
  $(x_1, \dots, x_k, x_{k+1}^{\mathrm{I}}, z_1, z_2, \dots) \in B$. We
  aim to show that $G^{P,k}$ is simple clopen again.

  Suppose first that $G$ is a clopen game of some fixed length
  $\omega \cdot n$, i.e., we consider $G$ as a game of length
  $\omega^2$ but only the first $\omega \cdot n$ moves are relevant
  for the payoff set. Fix some $k\in\omega$. In particular, $B$ is a clopen set in
  $\omega^{\omega \times \omega}$. We can naturally write
  $B = B_0 \times B_1 \times B_2$ for clopen sets
  $B_0 \subseteq (\BS)^k$, $B_1 \subseteq \BS$ and
  $B_2 \subseteq (\BS)^\omega$. By definition, the payoff set for the
  game $G^{P,k}$ is $B_0 \times (B_1)_{\mathrm{I}} \times B_2$, where
  \[(B_1)_{\mathrm{I}} = \{ x \in \BS : x^{\mathrm{I}} \in B_1 \},\]
  which is clopen. Moreover, in $G^{P,k}$ again only the first
  $\omega \cdot n$ moves are relevant, so $G^{P,k}$ is a clopen game of
  fixed length $\omega \cdot n$ and in particular simple clopen.

  Now suppose that $G$ is a simple clopen game obtained by condition
  \eqref{eq:DefSimpleI} in Definition \ref{DefSimple}, i.e., we are
  given simple clopen games $G_i$, $i<\omega$, and after Player I and
  II take turns producing $x_1, \dots, x_k, x_{k+1}$, Player I plays
  some natural number $i$ and the players continue according to the
  rules of $G_i$ (keeping the moves which produced
  $x_1, \dots, x_k, x_{k+1}$). That means for every $i < \omega$, some
  sequence $(x_1, \dots, x_k, x_{k+1}, i, z_1, z_2, \dots)$ is a
  winning run for Player I in $G$ iff
  $(x_1, \dots, x_k, x_{k+1}, z_1, z_2, \dots)$ is a winning run for
  Player I in $G_i$. We can assume inductively that the games $G^{P,k}_i$
  for $i<\omega$ (obtained as above) are simple clopen and we aim to
  show that $G^{P,k}_i$ is simple clopen. Consider the simple clopen game
  $G^*$ which is obtained by applying \eqref{eq:DefSimpleI} in
  Definition \ref{DefSimple} to the games $G^{P,k}_i$ and the natural
  number $k+1$. Suppose Players I and II produce
  $(x_1, \dots, x_k, x_{k+1}, i, z_1, z_2, \dots)$ in a run of
  $G^*$. Then
  \begin{align*}
    (x_1, \dots, x_k, &x_{k+1}, i, z_1, z_2, \dots) \text{ is a
                        winning run for Player I in } G^*\\
    \text{ iff }  (x_1, \dots, x_k, &x_{k+1}, z_1, z_2, \dots) \text{ is a
                                      winning run for Player I in } G_i^{P,k} \\
    \text{ iff } (x_1, \dots, x_k, &x_{k+1}^{\mathrm{I}}, z_1, z_2,
                                     \dots) \text{ is a winning run for Player I
                                     in } G_i\\
    \text{ iff } (x_1, \dots, x_k, &x_{k+1}^{\mathrm{I}}, i, z_1, z_2,
                                     \dots) \text{ is a winning run for Player I
                                     in } G\\
    \text{ iff } (x_1, \dots, x_k, &x_{k+1}, i, z_1, z_2, \dots) \text{ is a
                                     winning run for Player I in } G^{P,k},
  \end{align*}
  where the first equivalence holds by definition of $G^*$, the second
  equivalence holds by inductive hypothesis, the third equivalence by
  choice of $G$, and the fourth equivalence by definition of
  $G^{P,k}$. Hence $G^{P,k}$ and $G^*$ are equal and $G^{P,k}$ is a simple clopen
  game, as desired.

  The argument for simple clopen games obtained by condition
  \eqref{eq:DefSimpleII} in Definition \ref{DefSimple} is analogous.
\end{proof}

With Lemmata \ref{LemmaSCimpliesSP0} and \ref{LemmaSCimpliesSP},  Proposition \ref{PropositionSC} is proved. To finish the proof of Theorem \ref{main} one needs to show the following proposition. This is $(3) \Rightarrow (2)$ and $(1) \Rightarrow (2)$ in Theorem \ref{main}.

\begin{proposition}\label{PropositionSC2}
Suppose either that simple clopen games of length $\omega^2$ are determined or that $\sigma$-projective games of length $\omega$ are determined. Then simple $\sigma$-projective games of length $\omega^2$ are determined.
\end{proposition}
Proposition \ref{PropositionSC2} can be proved directly either by the method of the proof of Proposition
\ref{PropositionSC}, or by that of Proposition \ref{PropositionSimpleClopenEasyProof}. In the second case, one need only carry out a straightforward adaptation.
In the first case,
a simple $\sigma$-projective game $G$ is reduced
to the simple clopen game in which two players play the game $G$,
producing a sequence $x \in (\omega^\omega)^n$ for some
$n\in\omega$ such that there is a $\sigma$-projective set
$A \subset (\omega^\omega)^n$ with the property that Player I
wins $G$ iff $x \in A$, no matter how the players continue playing the
rest of $G$. After this, the players play the decoding game for $A$ to
determine who the winner is in a clopen way.\\

We close this section with a useful
characterization of the $\sigma$-projective sets, although it will not be used.

\begin{proposition}[Folklore]
  A set $A\subset\mathbb{R}$ is $\sigma$-projective if and only if
  it belongs to $L_{\omega_1}(\mathbb{R})$.
\end{proposition}
\begin{proof}
  Clearly, $L_{\omega_1}(\mathbb{R})$ is closed under countable sequences and
  $\mathcal{P}(\mathbb{R})\cap L_{\omega_1}(\mathbb{R})$ is closed under projections.

  Conversely, by induction on $\alpha<\omega_1$, one sees that $L_\alpha(\mathbb{R})$ and the satisfaction relation for $L_\alpha(\mathbb{R})$
  are coded by $\sigma$-projective
  sets of reals:
\begin{enumerate}
\item $L_0(\mathbb{R}) = \mathbb{R} = V_{\omega+1}$ is coded by
  itself. The satisfaction relation $S_0$ for $V_{\omega+1}$ is given by
  \[S_0(\phi, \vec a) \leftrightarrow V_{\omega+1}\models\phi(\vec
    a),\] for $\phi$ a formula in the language of set theory and
  $\vec a$ a finite sequence of reals.  Since every formula in the
  language of set theory is in the class $\Sigma_n$ for some $n$,
  $S_0$ belongs to the pointclass $\bigcup_{n<\omega}\Sigma^1_n$ (i.e.,
  the pointclass of countable unions of projective sets) and is thus
  $\sigma$-projective.
\item Suppose that a $\sigma$-projective code $C_\alpha$ for
  $L_\alpha(\mathbb{R})$ has been defined and that a
  $\sigma$-projective satisfaction predicate $S_\alpha$ for
  $L_{\alpha}(\mathbb{R})$ relative to $C_\alpha$ has been defined.  Suppose that
  $\vec a \in (C_\alpha)^n$ and
\[\phi = \exists y_1\, \forall y_2, \hdots Q y_m\, \phi_0(x_1,\hdots,
  x_n, y_1,\hdots, y_m)\]
  (where $Q$ is a quantifier and $\phi_0$ is $\Delta_0$) is a $\Sigma_m$-formula with $n$ free
variables. Then, letting
\[C^{\phi, \vec a}_\alpha = \big\{x\in\mathbb{R}: \exists y_1 \in
  C_\alpha\, \forall y_2\in C_\alpha\,\hdots\, Q y_m \in C_\alpha\,
  S_\alpha(\phi_0,\vec a, y_1,\hdots, y_m)\big\},\] a
$\sigma$-projective code $C_{\alpha+1}$ for $L_{\alpha+1}(\mathbb{R})$
can be defined by the disjointed union
\[C_{\alpha+1} = C_\alpha \, \dot\cup \, \{\hat C^{\phi,\vec
    a}_\alpha:\phi\text{ is a formula and }\vec a \in
  (C_\alpha)^n\},\] where $\hat C^{\phi,\vec a}_\alpha$ is a real
number coding\footnote{E.g., it could be the tuple
  $(\phi,\vec a,\alpha)$.} the set $C^{\phi,\vec a}_\alpha$.  A
satisfaction relation $S_{\alpha+1}$ can be defined from this: for
atomic formulae, we set
\begin{align*}
  S_{\alpha+1}(\cdot\in\cdot, x,y) \leftrightarrow 
  &\big(x,y\in C_\alpha \wedge S_\alpha(\cdot\in\cdot, x,y)\big) \vee\\
  &\big(x \in C_\alpha\wedge \exists \phi\,\exists\vec a \subset (C_\alpha)^{\text{lth}(a)}\,y=\hat C^{\phi,\vec a}_\alpha \wedge S_\alpha(\phi,\vec a, x)\big).
\end{align*}
For other formulae, this is done as above, so the satisfaction
relation $S_{\alpha+1}$ is seen to belong to the pointclass
$\bigcup_{n<\omega}\Sigma^1_n(C_{\alpha+1}, S_\alpha)$ and is thus by
using the inductive hypothesis $\sigma$-projective.
\item Suppose that a $\sigma$-projective code $C_\alpha$ for
  $L_\alpha(\mathbb{R})$ has been defined and that a
  $\sigma$-projective satisfaction predicate $S_\alpha$ for
  $L_{\alpha}(\mathbb{R})$ has been defined for every
  $\alpha<\lambda$, where $\lambda$ is a countable limit ordinal. Then
  $C_\lambda$ can be defined as the disjoint (countable) union of
  $C_{\alpha+1}\setminus C_\alpha$, for $\alpha<\lambda$. Thus
  $C_\lambda$ is $\sigma$-projective. The satisfaction relation
  $S_\lambda$ can be defined as above.
\end{enumerate}  
  This completes the proof.
\end{proof}

\section{Determinacy from large cardinals}\label{sec:simpleclopen}

In this section, we prove level-by-level that the existence of certain iterable
inner models with Woodin cardinals implies simple clopen determinacy of length $\omega^2$, which in turn implies $\sigma$-projective determinacy of length $\omega$. A different proof of this latter result can be found in \cite{Ag18}, but the proof we give here has the advantage that it requires almost no inner-model-theoretic background. Another advantage is that it easily generalizes to yield further results (cf. the next section).

The idea of this proof is to enhance a simple clopen game of length $\omega^2$ by presenting each player with a fine-structural model that can be manipulated to obtain information about the game. This method of proof is due to Neeman \cite{Ne04}; the difference in our context is that the models considered will contain many partial measures and, in addition to taking iterated ultrapowers, we will allow the players to remove end-segments of their models during the game so as to make the measures total and obtain new information. Although this determinacy proof could have been framed directly in terms of (the decoding game for) $\sigma$-projective sets, it seems somewhat more natural to consider simple clopen games of length $\omega^2$ instead and proceed by induction on the game rank.

Before we start with the proof, let us recall the relevant inner model
theoretic notions we will need. In what follows, we will work with
premice and formulae in the language of relativized premice
$\lpm = \{\dot \in, \dot E, \dot F, \dot x\}$, where $\dot E$ is a
predicate for a sequence of extenders, $\dot F$ is a
predicate for an extender, and $\dot x$ is a predicate for a real
number over which we construct the premouse.

For some real $x$, a \emph{potential $x$-premouse} is a model
$M = (J_\eta^{\vec{E}}, \in, \vec{E} \upharpoonright \eta, E_\eta,
x),$ where $\vec{E}$ is a fine extender sequence\footnote{See
  Definition $2.4$ in \cite{St10}, which goes back to Section $1$ in
  \cite{MS94} and \cite{SchStZe02}.} and $\eta$ an ordinal or $\eta = \Ord$ (see
Section $2$ in \cite{St10} for details). We say that such a potential
$x$-premouse $M$ is \emph{active} iff $E_\eta \neq \emptyset$. If
$\nu \leq \eta$, we write
$M | \nu = (J_\nu^{\vec{E}}, \in, \vec{E} \upharpoonright \nu, E_\nu,
x)$ for the corresponding initial segment of $M$. A potential
$x$-premouse $M$ is called an \emph{$x$-premouse} if every proper
initial segment of $M$ is $\omega$-sound. If it does not lead to
confusion, we will sometimes drop the $x$ and just call $M$ a
premouse. Informally, an \emph{$x$-mouse} is an iterable
$x$-premouse. We will avoid this term as the notion of iterability is
ambiguous, but $\omega_1$-iterability suffices for all our
arguments.\footnote{We will only need \emph{weak iterability}, in the sense of \cite{Ne04}.} Here we say that an $x$-premouse $M$ is
\emph{$\omega_1$-iterable} if it is iterable for countable stacks of
normal trees of length $< \omega_1$ (see Section $4.1$ of \cite{St10}
for a formal definition).

We consider premice belonging to various \emph{smallness classes}:

\begin{definition}
  Let $M$ be an $\omega_1$-iterable premouse and
  $\delta \in M \cap \Ord$.
\begin{enumerate}
\item We say $M$ is of class $S_0$ above $\delta$ if $M$ is a proper
  class or active.
\item We say $M$ is of class $S_{\alpha+1}$ above $\delta$ if there is
  some ordinal $\delta_0>\delta$ and some $N\trianglelefteq M$ with
  $\delta_0 < N \cap \Ord$ such that $N$ is of class $S_\alpha$ above
  $\delta_0$ and $\delta_0$ is Woodin in $N$.
\item Let $\lambda$ be a countable limit ordinal. We say $M$ is of
  class $S_\lambda$ above $\delta$ if $\lambda<\omega_1^M$ and $M$ is
  of class $S_\alpha$ above $\delta$ for all $\alpha<\lambda$.
\end{enumerate}

Moreover, we say $M$ is of class $S_\alpha$ if it is of class
$S_\alpha$ above $0$.
\end{definition}

\begin{example} Let $n \geq 1$ and recall that a premouse $M$ is
  called $n$-small if for every $\kappa$ which is a critical point of
  an extender on the sequence of extenders of $M$,
  $M|\kappa \nvDash \text{``there are }n\text{ Woodin
    cardinals''}$. Moreover, we say $M$ is $0$-small if $M$ is an
  initial segment of $L$. Let $M_n^\sharp$ denote the unique
  countable, sound, $\omega_1$-iterable premouse which is not
  $n$-small, but all of whose proper initial segments are $n$-small,
  if it exists and is unique. If $M_n^\sharp$ exists and is unique for
  all $n\in\omega$, then -- in particular -- every $x\in\BS$ has a
  sharp. Let $N_\omega^\sharp$ denote the smallest active premouse
  extending
  \[\bigcup_{n\in\omega}M_n^\sharp.\]
  Then $N_\omega^\sharp$ is of class $S_\omega$. If there is a
  premouse of class $S_{\omega+1}$, then it contains $N_\omega^\sharp$
  and in fact $N_\omega^\sharp$ is countable in it.
\end{example}

Note that if $M$ is a premouse of class $S_\alpha$, then $M$ is aware
of this, since $\alpha$ is countable in $M$. If $M$ is of class
$S_\alpha$ and not a proper class, then $M$ is active and we can
obtain a proper class model by iterating the
active extender of $M$ and its images out of the universe. By convention, we shall say that this proper class model is also of class $S_\alpha$.

We recall the definition of the game rank of a simple clopen game. 
If $G$ is a game of fixed length $\omega\cdot n$, then
$\gr(G) = n$. If $G$ is obtained from games $G_0, G_1, \hdots$, and
from an ordinal $\omega \cdot n$ as in Definition \ref{DefSimple}, we
let
\[ \gr(G) = \sup\{\gr(G_i)+\omega:i\in\omega\} + n. \] 

\begin{remark}
For every simple game $G$ and initial play $p$ of $G$, let $G_p$
  denote the rest of the game $G$ after $p$ has been played.
  Let
  $p_0, p_1, \dots$ be a sequence of initial plays of $G$ starting
  with the empty sequence $p_0 = \emptyset$ such that $p_{i+1}$
  end-extends $p_i$, and either $\gr(G_{p_i})$ is a successor ordinal
  with $\gr(G_{p_i}) = \gr(G_{p_{i+1}}) +1$ or $\gr(G_{p_i})$ is a
  limit ordinal. Then this sequence has to be of finite length $k+1$
  and we can choose $k$ large enough such that $\gr(G_{p_k}) = 1$.
\end{remark}


\begin{theorem}\label{prop:simpleclopen}
  Suppose that $\gamma < \omega_1$ and for every $y \in \BS$ there is
  an $x \geq_T y$ such that there is a proper class $x$-premouse of
  class $S_{\gamma}$ which is a model of $\ZFC$. Then every simple
  clopen game $G$ such that $\gr(G)\leq\gamma$ is determined.
\end{theorem}

In the proof, we are going to use premice of class $S_\gamma$ to apply
the following theorem of Neeman multiple times. 

\begin{theorem}[Neeman, Theorem 2A.2 in \cite{Ne04}]\label{thm:Neeman}
There are binary formulae $\phi_{\mathrm{I}}$ and
  $\phi_{\mathrm{II}}$ in the language of set theory such that the following hold for any transitive, weakly iterable\footnote{As defined in
    Appendix A, Iterability in \cite{Ne04}. Note that
    $\omega_1$-iterability implies weak iterability.} premouse $M$ which is a model of
  $\ZFC$:
  
  Let $\delta<\omega_1$ be a Woodin cardinal in $M$ and let
  $\dot A\in M$ and $\dot B\in M$ be $\Col(\omega,\delta)$-names for a
  subset of $\BS$. Then,
\begin{enumerate}
\item If $M\models\phi_{\mathrm{I}}[\delta, \dot A]$, there is a strategy
  $\sigma$ for Player I in a game of length $\omega$ such that
  whenever $x$ is a play by $\sigma$, there is a non-dropping iterate $N$ of $M$ with embedding $j$
  and an $N$-generic $g\subset\Col(\omega,j(\delta))$ such that
  $x\in N[g]$ and $x \in j(\dot A)[g]$.
\item If $M\models\phi_{\mathrm{II}}[\delta, \dot B]$, there is a strategy
  $\tau$ for Player II in a game of length $\omega$ such that whenever
  $x$ is a play by $\tau$, there is a non-dropping iterate $N$ of $M$ with embedding $j$ and an
  $N$-generic $g\subset\Col(\omega,j(\delta))$ such that $x\in N[g]$ and
  $x \in j(\dot B)[g]$.
\item Otherwise, there is an $M$-generic $g\subset\Col(\omega,\delta)$
  and an $x\in M[g]$ such that $x\not\in \dot A[g]$ and
  $x\not\in \dot B[g]$.
\end{enumerate}
\end{theorem}

We now proceed to:

\begin{proof}[Proof of Theorem \ref{prop:simpleclopen}]
  We may as well assume that $\gamma$ is infinite, since otherwise $G$
  has fixed length ${<}\omega^2$, so it is determined by the results
  of \cite{Ne95} and \cite{Ne02} (see also the introduction of
  \cite{Ne04}).

  Let $G$ be a simple clopen game with $\gr(G) \leq \gamma$, say
  $\gr(G) = \alpha$. We assume that $\alpha$ is a successor ordinal,
  since the limit case is similar. Let $r \in\BS$ code all parameters
  used in the definition of $G$. We may as well assume that $r$
  belongs to the Turing cone given by the hypothesis of the
  theorem.

  To each non-terminal play $p$ of $G$ corresponds a game
  \[ G_p := \text{ the game $G$ after $p$ has been played.} \]
  Clearly, $G_p$ is a simple game and $\gr(G_p) \leq\alpha$. Given $y \geq_T r$, a $y$-premouse $M$, we select $\delta\in\Ord^M$ and define formulae $\phi^p_{\mathrm{I}}$
  and $\phi^p_{\mathrm{II}}$, and two sets $\dot W^p_{\mathrm{I}}$ and
  $\dot W^p_{\mathrm{II}}$, each of which is either a $\Col(\omega,\delta)$-name for a set of real numbers or a set of natural numbers. The definition is by
  induction on $\gr(G_p)$ (not on $p$!). Formally, these names and
  sets of course depend on the model $M$ in which they are defined, so
  we sometimes write $\dot W^p_{\mathrm{I}}(M)$ and
  $\dot W^p_{\mathrm{II}}(M)$ to make this explicit. But for
  simplicity and readability we will omit this whenever it does not
  lead to confusion. 

  At the same time, we show that there are names, names for names,
  etc., for these sets such that their interpretation with respect to
  generics $g_0, \dots, g_n$ for some $n\in\omega$ for collapsing Woodin cardinals of $M$ is
  a set of the same form with respect to the model $M[g_0]\dots[g_n]$
  instead of $M$. 
  More precisely, if $M$ is a model as above and $M$ has (at least) $n+1$ Woodin cardinals, then let
  $\mathbb{P}_{n} = \mathbb{P}(\delta_0, \dots, \delta_n)$ be the
  forcing iteration of length $n+1$ collapsing the ordinals
  $\delta_0, \delta_1, \dots, \delta_n$ to $\omega$ one after the
  other, where $\delta_0$ is the least Woodin cardinal in $M$ and
  $\delta_{i+1}$ is the least Woodin cardinal in $M$ above $\delta_i$
  for all $0 \leq i < n$, i.e.,
  $\mathbb{P}_0 = \Col(\omega,\delta_0)$ and for all $0 \leq i < n$,
  \[ \mathbb{P}_{i+1} = \mathbb{P}_{i} *
    \Col(\omega,\delta_{i+1})^{\check{} },\] where
  $\Col(\omega,\delta_{i+1})^{\check{} } \in M$ is the canonical
  $\mathbb{P}_i$-name for $\Col(\omega,\delta_{i+1})$. This is of course equivalent to the product and to the $\Col(\omega,\delta_n)$, but we are interested in the step-by-step collapse.
  
  We consider nice $\mathbb{P}_n$-names $\dot p$ for finite sequences of reals with the property that the game rank of $G_{\dot p}$ is decided by the empty condition.\footnote{Note that the game rank of $G_p$ depends only on $G$ and finitely many digits of $p$.}
  By induction on the game rank of $G_{\dot p}$, we 
  will specify for every $n < \omega$ such that $M$ has at least $n+1$ Woodin cardinals and every such
  $\mathbb{P}_n$-name $\dot p$,
a $\mathbb{P}_n$-name $\dot B^{\dot p,I}_n$ in $M$ such that whenever $G$ is
  $\mathbb{P}_n$-generic over $M$,
  \[ \dot B^{\dot p,I}_n[G] =  \dot W^p_{\mathrm{I}}(M[G]), \] where
  $p = \dot p[G]$. We
  will call this name $\dot B^{\dot p,I}_n$ the \emph{good $\mathbb{P}_n$-name
    for $\dot W^p_{\mathrm{I}}$}. We will also define analogous
  names $\dot B^{\dot p,II}$ for $\dot W^p_{\mathrm{II}}$. We now proceed to the recursive definition of $\phi^p_{\mathrm{I}}$, $\phi^p_{\mathrm{II}}$, $\dot W^p_{\mathrm{I}}$, $\dot W^P_{\mathrm{II}}$, $\dot B_n^{p,I}$, and $\dot B_n^{p,II}$.\footnote{For every $p$, $\phi^p_{\mathrm{I}}$ will be one of three formulae (there is one possibility for the successor case and two for the limit case). We will however use the notation $\phi^p_{\mathrm{I}}$ to make the presentation uniform. Similarly for $\phi^p_{\mathrm{II}}$.}

\begin{case}
  $\gr(G_p) = 2$ (this is the base case).
\end{case}

Let $y \geq_T r$ and let $M$ be any $y$-premouse which is a model of
$\ZFC$ and of class $S_1$. Assume without loss of generality that $M$ is minimal, in the sense that no proper initial segment of $M$ is a model of $\ZFC$ of class $S_1$. Let $\delta$ be the least Woodin cardinal in $M$,
so that $M$ is of class $S_0$ above $\delta$, and let $p \in M$. We define the
$\Col(\omega,\delta)$-names for sets of reals $\dot W^p_{\mathrm{I}}$
and $\dot W^p_{\mathrm{II}}$ by\footnote{Here and below, all names for real are assumed to be nice.}
\[\dot W^p_{\mathrm{I}} = \dot W^p_{\mathrm{I}}(M) = \{(\dot y,q) \, | \, q
  \forces{M}{\Col(\omega,\delta)} \text{ Player I has a winning
    strategy in } G_{\check p^\frown \dot y} \}\] and
\[\dot W^p_{\mathrm{II}} = \dot W^p_{\mathrm{II}}(M) = \{(\dot y,q) \, | \, q
  \forces{M}{\Col(\omega,\delta)} \text{ Player II has a winning
    strategy in } G_{\check p^\frown \dot y} \}.\] Let
$\phi^p_{\mathrm{I}}$ and $\phi^p_{\mathrm{II}}$ be the formulae given
by Neeman's theorem (Theorem \ref{thm:Neeman}) such that the following
hold:
\begin{enumerate}
\item If $M\models\phi^p_{\mathrm{I}}[\dot W^p_{\mathrm{I}}]$,
  there is a strategy $\sigma$ for Player I in a game of length
  $\omega$ such that whenever $x$ is a play by $\sigma$, there is an
  iterate $N$ of $M$, an elementary embedding
  $j \colon M \rightarrow N$ and an $N$-generic
  $h\subset\Col(\omega,j(\delta))$ such that $x\in N[h]$ and
  $x \in j(\dot W^p_{\mathrm{I}})[h]$.
\item If
  $M\models\phi^p_{\mathrm{II}}[\dot W^p_{\mathrm{II}}]$,
  there is a strategy $\tau$ for Player II in a game of length
  $\omega$ such that whenever $x$ is a play by $\tau$, there is an
  iterate $N$ of $M$, an elementary embedding
  $j\colon M \rightarrow N$ and an $N$-generic
  $h\subset\Col(\omega,j(\delta))$ such that $x\in N[h]$ and
  $x \in j(\dot W^p_{\mathrm{II}})[h]$.
\item Otherwise, there is an $M$-generic
  $g \subset \Col(\omega,\delta)$ and an $x\in M[g]$ such that
  $x\not\in \dot W^p_{\mathrm{I}}[g]$ and
  $x\not\in \dot W^p_{\mathrm{II}}[g]$.
\end{enumerate}

Now, we specify the good $\mathbb{P}_n$-names $\dot B^{\dot p,I}_n$ as claimed
above, for models with enough Woodin cardinals below $\delta$ so that $\mathbb{P}_n$ is defined. 
For $n = 0$, suppose that $M$ is a premouse with Woodin cardinals $\delta_0<\delta$ and is of class $S_0$ above $\delta$ and minimal above $\delta$.\footnote{A case of interest will be that of such $M$ which are initial segments  of a model $N$ of some class $S_\beta$ which is minimal above $\delta_0$. In such an $N$, $\delta$ need not be a cardinal.} 
Let $\dot p$ be a nice $\Col(\omega, \delta_0)$-name for a
finite sequence of reals and set
\[ \dot B^{\dot p,I}_0 = \{ ((\dot{\dot{y}}, \dot q), q_0) \, | \, q_0
  \forces{}{\Col(\omega,\delta_0)} \dot q
  \forces{}{\Col(\omega,\delta)^{\check{}}} \text{ Player I has a
    winning strategy in } G_{\check{\dot{p}}^\frown \dot{\dot{y}}}
  \}. \] The good $\mathbb{P}_n$-names $\dot B^{\dot p,I}_n$ for $n > 0$ and similar
names $\dot B^{\dot p,II}_n$ for $\dot W^p_{\mathrm{II}}$ are defined analogously.


\begin{case}
  $\gr(G_p) = \gamma+1$ for some $\gamma\geq 2$.
\end{case}

Let $y \geq_T r$ and let $M$ be any $y$-premouse which is a model of
$\ZFC$ of class $S_{\gamma +1}$. Assume without loss of generality that $M$ is minimal, in the sense that no proper initial segment of $M$ is a model of $\ZFC$ of class $S_{\gamma+1}$. Let $\delta$ be the least Woodin cardinal in
$M$, so that $M$ is of class $S_\gamma$ above $\delta$, and let $p \in M$. Then let
\[\dot W^p_{\mathrm{I}} = \dot W^p_{\mathrm{I}}(M) = \{(\dot y,q) \, | \, q
  \forces{M}{\Col(\omega)} \phi^{\check p^\frown \dot
    y}_{\mathrm{I}}[\dot B] \},\] where 
 $\dot B \in M$ is the good $\Col(\omega,\delta)$-name with respect to
  $\check{p}^\frown \dot y$, so that whenever $G$ is
  $\Col(\omega, \delta)$-generic over $M$,
  $\dot B[G] = \dot W^{p^\frown y}_{\mathrm{I}}(M[G])$,
  where $p^\frown y = (\check{p}^\frown \dot y)[G]$. 
We also let
\[\dot W^p_{\mathrm{II}} = \dot W^p_{\mathrm{II}}(M) = \{(\dot y,q) \, | \, q
  \forces{M}{\Col(\omega,\delta)} \phi^{\check p^\frown \dot
    y}_{\mathrm{II}}[\dot B] \},\] for 
$\dot B$ analogous as above. Note that $\dot B$ is a good name for
$\dot W^{p^\frown y}_{\mathrm{I}}$ or
$\dot W^{p^\frown y}_{\mathrm{II}}$ respectively and hence already
defined since for every real $y$, $\gr(G_{p^\frown y}) = \gamma$. Let
$\phi^p_{\mathrm{I}}$ and $\phi^p_{\mathrm{II}}$ be the formulae given
by Neeman's theorem (Theorem \ref{thm:Neeman}) such that the following
hold:
\begin{enumerate}
\item If $M\models\phi^p_{\mathrm{I}}[\dot W^p_{\mathrm{I}}]$,
  there is a strategy $\sigma$ for Player I in a game of length
  $\omega$ such that whenever $x$ is a play by $\sigma$, there is an
  iterate $N$ of $M$, an elementary embedding
  $j \colon M \rightarrow N$ and an $N$-generic
  $h\subset\Col(\omega,j(\delta))$ such that $x\in N[h]$ and
  $x \in j(\dot W^p_{\mathrm{I}})[h]$.
\item If
  $M\models\phi^p_{\mathrm{II}}[\dot W^p_{\mathrm{II}}]$,
  there is a strategy $\tau$ for Player II in a game of length
  $\omega$ such that whenever $x$ is a play by $\tau$, there is an
  iterate $N$ of $M$, an elementary embedding
  $j\colon M \rightarrow N$ and an $N$-generic
  $h\subset\Col(\omega,j(\delta))$ such that $x\in N[h]$ and
  $x \in j(\dot W^p_{\mathrm{II}})[h]$.
\item Otherwise, there is an $M$-generic $g \subset \Col(\omega,\delta)$
  and an $x\in M[g]$ such that $x\not\in \dot W^p_{\mathrm{I}}[g]$ and
  $x\not\in \dot W^p_{\mathrm{II}}[g]$.
\end{enumerate}

Now, we specify the good $\mathbb{P}_n$-names $\dot B^{\dot p,I}_n$. For $n = 0$, suppose $M$ is a premouse with Woodin cardinals $\delta_0<\delta$ and that $M$ is of class $S_\gamma$ above $\delta$ and minimal above $\delta$. Let $\dot p \in M$ be a nice $\Col(\omega, \delta_0)$-name
for a finite sequence of reals. Set 
\[ \dot B^{\dot p,I}_0 = \{ ((\dot{\dot{y}}, \dot q), q_0) \, | \, q_0
  \forces{}{\Col(\omega,\delta_0)} \dot q
  \forces{}{\Col(\omega,\delta)^{\check{}}}
  \phi^{\check{\dot{p}}^\frown \dot{\dot{y}}}_{\mathrm{I}}[
  \dot B] \}, \] where $\dot B$ is a good
$\mathbb{P}(\delta_0,\delta)$-name for $\dot W_{\mathrm{I}}^p$. That
means $\dot B$ is such that whenever $G$ is
$\mathbb{P}(\delta_0,\delta)$-generic over $M$,
$\dot B[G] = \dot W^{p^\frown y}_{\mathrm{I}}(M[G])$, where
$p^\frown y = (\check{\dot{p}}^\frown \dot{\dot{y}})[G]$. The
$\mathbb{P}_n$-names $\dot B^{\dot p,I}_n$ for $n > 0$ and similar names $\dot B^{\dot p,II}_n$ for
$\dot W^p_{\mathrm{II}}$ are defined analogously.

\begin{case}
  $\gr(G_p) = \lambda$ is a limit ordinal and the rules of $G$ dictate
  that, after $p$, it is Player I's turn.
\end{case}

Let $y \geq_T r$ and let $M$ be any $y$-premouse which is a model of
$\ZFC$ and of class $S_\lambda$. Let $p \in M$. Then let
$\dot W^p_{\mathrm{I}} = \dot W^p_{\mathrm{I}}(M)$ be the set
of all $k\in\omega$ such that there is an $\eta < M \cap \Ord$ such
that $M|\eta$ is an active $y$-premouse of class
$S_{\gr(G_{p^\frown k})}$ which is minimal, in the sense that no proper initial segment thereof is of class $S_{\gr(G_{p^\frown k})}$,
and, if we let $M^*$ be the result of iterating the active extender of
$M|\eta$ out of the universe, then
\[M^* \models \phi^{p^\frown k}_{\mathrm{I}}[\delta^*, \dot
  W_{\mathrm{I}}^{p^\frown k}(M^*)].\] This makes sense because for all
$k\in\omega$, $\gr(G_{p^\frown k}) < \gr(G_{p})$, so the set
$\dot W_{\mathrm{I}}^{p^\frown k}(M^*)$ has been defined for
all $M^*$ as above.  Moreover, $\dot W^p_{\mathrm{I}}$ belongs to $M$,
since the formulae
$\phi^{p^\frown k}_I(\dot W_{\mathrm{I}}^{p^\frown
  k}[M^*])$, together with their parameters
$\dot W_{\mathrm{I}}^{p^\frown k}(M^*)$, are
definable\footnote{Recall that the structure $M$ includes a predicate
  for its extender sequence.} uniformly in $p$, $\gr(G_p)$, and $\eta$, as can be shown inductively by following this
construction and the proof of Theorem \ref{thm:Neeman} (cf. the proof
of \cite[Theorem 1E.1]{Ne04}).

We let $\phi^p_{\mathrm{I}}[\dot W_{\mathrm{I}}^p(M)]$
be the formula asserting that $\dot W^p_{\mathrm{I}}$ is
non-empty. 
Similarly, let
$\dot W^p_{\mathrm{II}} = \dot W^p_{\mathrm{II}}(M)$ be the set
of all $k\in\omega$ such that whenever $\eta < M \cap \Ord$ is
such that $M|\eta$ is an active $y$-premouse of class
$S_{\gr(G_{p^\frown k})}$ which is minimal, then
\[M^* \models \phi^{p^\frown k}_{\mathrm{II}}[\dot
  W_{\mathrm{II}}^{p^\frown k}(M^*)],\] for $M^*$ defined as above. We let
$\phi^p_{\mathrm{II}}[\dot W_{\mathrm{II}}^p(M)]$ be the
formula asserting that $\dot W^p_{\mathrm{II}}$ is equal to
$\omega$. As before, the set $\dot W^p_{\mathrm{II}}$ belongs to $M$.

The good $\mathbb{P}_n$-names $\dot B^{\dot p,I}_n$ and $\dot B^{\dot p,II}_n$ are defined as before, for premice $M$ which are of class $S_\lambda$ above finitely many Woodin cardinals.

\begin{case}
  $\gr(G_p) = \lambda$ is a limit ordinal and the rules of $G$ dictate
  that, after $p$, it is Player II's turn. 
\end{case}

Let $y \geq_T r$ and let $M$ be any $y$-premouse which is a model of
$\ZFC$ and of class $S_\lambda$. Let $p \in M$. Then let
$\dot W^p_{\mathrm{I}}(M)$ be the set of all $k\in\omega$ such
that there is an $\eta < M \cap \Ord$ such that $M|\eta$ is an active
$y$-premouse of class $S_{\gr(G_{p^\frown k})}$ which is minimal, in the sense that no proper initial segment of $M|\eta$ is of class $S_{\gr(G_{p^\frown k})}$, and, if we again let $M^*$ be the result of
iterating the active extender of $M|\eta$ out of the universe, then
\[ M^* \models \phi^{p^\frown k}_{\mathrm{I}}[\dot
  W_{\mathrm{I}}^{p^\frown k}(M^*)]. \] Again, $\dot W_{\mathrm{I}}^p$ is in
$M$. We let
$\phi^p_{\mathrm{I}}[\dot W_{\mathrm{I}}^p(M)]$ be the
formula asserting that $\dot W^p_{\mathrm{I}}$ is equal to
$\omega$. Similarly, let
$\dot W^p_{\mathrm{II}} = \dot W^p_{\mathrm{II}}(M)$ be the set
of all $k\in\omega$ such that whenever $\eta < M \cap \Ord$ is such
that $M|\eta$ is an active $y$-premouse of class
$S_{\gr(G_{p^\frown k})}$ which is minimal, then
\[M^* \models \phi^{p^\frown k}_{\mathrm{II}}[\dot
  W_{\mathrm{II}}^{p^\frown k}(M^*)],\] for $M^*$ defined as above. We let
$\phi^p_{\mathrm{II}}[\dot W_{\mathrm{II}}^p(M)]$ be
the formula asserting that $\dot W^p_{\mathrm{II}}$ is non-empty.

The good $\mathbb{P}_n$-names $\dot B^{\dot p,I}_n$ and $\dot B^{\dot p,II}_n$ are defined as before, for premice $M$ which are of class $S_\lambda$ above finitely many Woodin cardinals.\\

This completes the definition of the names $\dot W^p_{\mathrm{I}}$,
$\dot W^p_{\mathrm{II}}$, $\dot B^{\dot p,I}_n$, $\dot B^{\dot p,II}_n$, and the formulae
$\phi^p_{\mathrm{I}}$ and $\phi^p_{\mathrm{II}}$. Continuing with the
proof of Theorem \ref{prop:simpleclopen}, we prove a technical
claim which shows that these names behave well under elementary
embeddings.

\begin{claim}\label{cl:nameselememb}
  Suppose $\gr(G)$ is a successor ordinal, let $M$ be a proper class
  premouse which is a model of $\ZFC$ of class $S_{\gr(G)}$ and minimal, in the sense that no proper initial segment of $M$ is a model of $\ZFC$ of class $S_{\gr(G)}$.\footnote{Since $\gr(G)$ is a successor ordinal and $M$ is minimal, $M$ has a Woodin cardinal.} Let
  $j \colon M \rightarrow N$ be an elementary embedding. Let $p \in M$
  be a finite sequence of reals in $M$. Then
  \[ j(\dot W_{\mathrm{I}}^p(M)) = \dot W_{\mathrm{I}}^p(N
) \] and, analogously,
  \[ j(\dot W_{\mathrm{II}}^p(M)) = \dot W_{\mathrm{II}}^p(N). \]
\end{claim}
\begin{proof} 
Let $\delta$ denote the least Woodin cardinal of $M$.
We will in fact show that whenever $\dot p \in M$ is a
  $\mathbb{P}_n$-name for a finite sequence of reals such that $\gr(G_{\dot p})$ is decided by the empty condition, we have
\[j(\dot B^{\dot p,I}_n) = (\dot B^{j(\dot p),I}_n)^N,\] 
and
\[j(\dot B^{\dot p,II}_n) = (\dot B^{j(\dot p),II}_n)^N,\]  
i.e., if
  $\dot B \in M$ is a good $\mathbb{P}_n$-name such that whenever $G$
  is $\mathbb{P}_n$-generic over $M$,
  \[ \dot B[G] = \dot W_{\mathrm{I}}^{\dot p[G]}(M[G]), \]
  then $j(\dot B) \in N$ is a good $j(\mathbb{P}_n)$-name such that
  whenever $H$ is $j(\mathbb{P}_n)$-generic over $N$,
  \[ j(\dot B)[H] = \dot W_{\mathrm{I}}^{j(\dot p)[H]}(N[H]) \] 
    and similarly for $\dot B^{\dot p, II}_n$.

  This will yield the claim if applied to $n = 0$, as, by definition,
  \[\dot W_{\mathrm{I}}^p(M) = \{(\dot y, q):q \forces{
      M}{\Col(\omega,\delta)} \phi^{\check p^\frown \dot
      y}_{\mathrm{I}}[\dot B^{\check p^\frown \dot
      y, I}_0] \}.\] 
       Hence,
  \begin{align*}
    j(\dot
    W_{\mathrm{I}}^p(M)) 
    &=    \{(\dot y, q):q \forces{
      N}{\Col(\omega,j(\delta))} \phi^{\check p^\frown \dot
      y}_{\mathrm{I}}[ j(\dot B^{\check{p}^\frown \dot y,I}_0)]
    \} \\
& = 
\{(\dot y, q):q \forces{
      N}{\Col(\omega,j(\delta))} \phi^{\check p^\frown \dot
      y}_{\mathrm{I}}[  (\dot B_0^{j(\check p^\frown \dot y),I})^N]
    \} \\    
& = \dot W_{\mathrm{I}}^{p}(N),
  \end{align*}
  and similarly for $\dot B^{\dot p, II}_n$.

  By the definition of simple games, $\gr(G_p)$ depends only on $G$,
  the length of $p$, and finitely many values of $p$ (those in which
  the players determine the subgames played). Thus, if $\dot p$
  is a $\mathbb{P}_n$-name in $M$ for a finite sequence of reals such
  that $\forces{M}{\mathbb{P}_n} \gr(G_{\dot p}) = \gamma +1$ for some
  countable ordinal $\gamma$ and if $\dot y$ is a
  $\mathbb{P}(\delta_0, \dots, \delta_n, \delta)$-name for a real,
  then
  \[\forces{M}{\mathbb{P}(\delta_0, \dots, \delta_n, \delta)}
  \gr(G_{\check{\dot{p}}^\frown \dot y}) = \gamma.\] 
The claim is proved simultaneously for all premice $M$ and all $\dot p \in M$, by induction on the game rank of $\gr(G_p)$. We prove it for the names $\dot B^{\dot p,I}_n$ for Player I; the other part is similar. We proceed by cases: 


  First suppose that $\gamma = 2$ and let $n<\omega$ and $\dot p \in M$ be a
  $\mathbb{P}_n$-name for a finite sequence of reals such that
  $\forces{M}{\mathbb{P}_n} \gr(G_{\dot p}) = 2$. Let $\dot B = \dot B^{\dot p,I}_n \in M$, so that whenever $G$ is
  $\mathbb{P}_n$-generic over $M$,
\[ \dot B[G] = \dot W_{\mathrm{I}}^{\dot p[G]}(M[G]). \]
By definition, \[ \forces{M}{} \dots \forces{}{} \dot B = \{(\dot y,q) \, | \, q
  \forces{}{\Col(\omega,\delta)} \text{ Player I has a winning
    strategy in } G_{\check{\dot{p}}^\frown \dot y} \}. \]
By elementarity,
\[ \forces{N}{} \dots \forces{}{} j(\dot B) = \{(\dot y,q) \, | \, q
  \forces{}{\Col(\omega,j(\delta))} \text{ Player I has a winning
    strategy in } G_{j(\check{\dot{p}})^\frown \dot y} \}. \] But this
implies that $j(\dot B)$ is a good $j(\mathbb{P}_n)$-name such that
whenever $H$ is $j(\mathbb{P}_n)$-generic over $N$,
\[ j(\dot B)[H] = \dot W_{\mathrm{I}}^{j(\dot p)[H]}(N[H]). \]
This finishes the case $\gamma = 2$.

If $\gamma = \beta+1$, where $\beta >1$, let again $n<\omega$, $\dot p \in M$ be a
$\mathbb{P}_n$-name for a finite sequence of reals such that
$\forces{M}{\mathbb{P}_n} \gr(G_{\dot p}) = \beta+1$ and let
$\dot B = \dot B^{\dot p,I}_n\in M$, so that whenever $G$ is
$\mathbb{P}_n$-generic over $M$,
  \[ \dot B[G] = \dot W_{\mathrm{I}}^{\dot p[G]}(M[G]). \]
  This means that
  \[ \forces{M}{} \dots \forces{}{} \dot B = \{(\dot y,q) \, | \, q
    \forces{}{\Col(\omega,\delta)} \phi^{\check{\dot{p}}^\frown \dot
      y}_{\mathrm{I}}[\dot C]\}, \] where $\dot C \in M$
  is a good $\mathbb{P}(\delta_0, \dots, \delta_n, \delta)$-name such that
  whenever $g$ is
  $\mathbb{P}(\delta_0, \dots, \delta_n, \delta)$-generic over $M$,
  \[ \dot C[g] = \dot W_{\mathrm{I}}^{(\check{\dot{p}}^\frown \dot y)[g]}(M[g]). \] By inductive hypothesis, and using that
  $\forces{M}{\mathbb{P}(\delta_0, \dots, \delta_n, \delta)}
  \gr(G_{\check{\dot p}^\frown \dot y}) = \beta < \gr(G_{\check{\dot{p}}})$, we have that
  $j(\dot C) \in N$ is a good
  $\mathbb{P}(j(\delta_0), \dots, j(\delta_n), j(\delta))$-name such
  that whenever $h$ is
  $\mathbb{P}(j(\delta_0), \dots, j(\delta_n), j(\delta))$-generic
  over $N$,
  \[ j(\dot C)[h] = \dot W_{\mathrm{I}}^{j(\check{\dot{p}}\frown \dot
      y)[h]}(N[h]). \] Since, by elementarity, 
  \[ \forces{N}{} \dots \forces{}{} j(\dot B) = \{(\dot y,q) \, | \, q
    \forces{}{\Col(\omega,j(\delta))} \phi^{j(\check{\dot{p}}^\frown
      \dot y)}_{\mathrm{I}}[j(\dot C)]\}, \] it follows
  that $j(\dot B)$ is a good $j(\mathbb{P}_n)$-name such that for
  every $j(\mathbb{P}_n)$-generic $H$ over $N$,
  \[ j(\dot B)[H] = \dot W_{\mathrm{I}}^{j(\dot p)[H]}(N[H]), \] as desired. 

  If $\gamma$ is a limit ordinal, let again $n<\omega$, $\dot p \in M$ be a
  $\mathbb{P}_n$-name for a finite sequence of reals such that
  $\forces{M}{\mathbb{P}_n} \gr(G_{\dot p}) = \gamma$. Let
  $\dot B = \dot B^{\dot p,I}_n\in M$, so that whenever $G$
  is $\mathbb{P}_n$-generic over $M$,
  \[ \dot B[G] = \dot W_{\mathrm{I}}^{\dot p[G]}(M[G]). \]
  This means that
  \begin{align*}
    \forces{M}{} \dots & \forces{}{} \dot B = \{ k \in \omega \, | \,
    \text{there is an initial segment } M^* \text{ minimal of class } \\ & 
    S_{\gr(G_{\dot p^\frown k})}  \text{ which satisfies } \phi^{\dot p^\frown
      k}_\mathrm{I}[\dot C] \}, 
  \end{align*}
  where $\dot C \in M^*$ is a good $\mathbb{P}_n$-name such that whenever
  $g$ is $\mathbb{P}_n$-generic over $M^*$,
  \[ \dot C[g] = \dot W_{\mathrm{I}}^{(\dot p^\frown k)[g]}(M^*[g]). \] By inductive hypothesis,
  $j(\dot C) \in N|(j(M^*) \cap \Ord)$ is a good $j(\mathbb{P}_n)$-name
  such that whenever $h$ is $j(\mathbb{P}_n)$-generic over
  $N|(j(M^*) \cap \Ord)$,
  \[ j(\dot C)[h] = \dot W_{\mathrm{I}}^{j(\dot p^\frown
      k)[h]}(N|(j(M^*) \cap \Ord)[h]). \] Since, by
  elementarity,
  \begin{align*}
    \forces{N}{} \dots & \forces{}{} j(\dot B) = \{ k \in \omega \, | \,
    \text{there is an initial segment } N^* \text{ minimal of class } 
\\ &  S_{\gr(G_{\dot p^\frown k})} \text{ which satisfies } \phi^{j(\dot p^\frown
      k)}_\mathrm{I}[ j(\dot C)] \},
  \end{align*}
  it follows that $j(\dot B)$ is a good $j(\mathbb{P}_n)$-name such
  that for every $j(\mathbb{P}_n)$-generic $H$ over $N$,
  \[ j(\dot B)[H] = \dot W_{\mathrm{I}}^{j(\dot p)[H]}(N[H]), \] as desired. 

  The argument for $\dot W_{\mathrm{II}}^p(M)$ is
  analogous. This completes the proof of the claim.
\end{proof}

Now we turn to the proof of the following claim, from which the
theorem follows. Recall that $\alpha = \gr(G)$.

\begin{claim}\label{cl:main}
  Let $M$ be an $r$-premouse which of $\ZFC$ of class $S_\alpha$ but has no proper initial segment of
  class $S_{\alpha}$. Let $\delta$ denote the least Woodin cardinal
  in $M$. Then
  \begin{enumerate}
  \item \label{proofsd1} If
    $M \models \phi^\emptyset_{\mathrm{I}}[\dot
    W_{\mathrm{I}}^\emptyset]$, then Player I has a winning strategy
    in $G$.
  \item \label{proofsd2} If
    $M \models \phi^\emptyset_{\mathrm{II}}[\dot
    W_{\mathrm{II}}^\emptyset]$, then Player II has a winning strategy
    in $G$.
  \item \label{proofsd3}
    $M\models\phi^\emptyset_{\mathrm{I}}[\dot
    W_{\mathrm{I}}^\emptyset] \vee
    \phi^\emptyset_{\mathrm{II}}[ \dot
    W_{\mathrm{II}}^\emptyset]$.
  \end{enumerate}
\end{claim}
\begin{proof}
  If $M$ is active, then we may 
  identify it with the proper class sized iterated ultrapower by
  its top extender and its images, so we may assume that $M$ is a
  model of $\ZFC$.  Note that $M$ always has a Woodin 
  cardinal, as $\alpha$ is a successor ordinal. We first prove
  \eqref{proofsd3}. Suppose that
  $M\not\models\phi^\emptyset_{\mathrm{I}}[\dot
  W_{\mathrm{I}}^\emptyset]$ and
  $M\not\models\phi^\emptyset_{\mathrm{II}}[\dot
  W_{\mathrm{II}}^\emptyset]$. We inductively construct a run of the
  game $G$ by its initial segments $p_m$ together with reals $y_m$,
  $y_m$-premice $M^{p_m}$ which are proper class models of $\ZFC$ and
  ordinals $\delta_{p_m}$. Let $p_0$ be the empty play, $y_0 = r$, and
  $M_0 = M$. Inductively, suppose that $p_m$, $y_m$, and $M^{p_m}$
  have been defined and that
  \[ M^{p_m}\not\models \phi^{p_m}_{\mathrm{I}}[\dot
    W_{\mathrm{I}}^{p_m}] \text{ and } M^{p_m} \not\models
    \phi^{p_m}_{\mathrm{II}}[\dot
    W_{\mathrm{II}}^{p_m}], \]
    and let $\delta_{p_m}$ be the least Woodin cardinal of $M^{p_m}$, if it exists.

  If $\gr(G_{p_m}) = \gamma +1$ is a successor ordinal for some
  $\gamma \geq 2$, then, by Neeman's theorem (Theorem
  \ref{thm:Neeman}), there is an $M^{p_m}$-generic
  $g_m \subset\Col(\omega,\delta_{p_m})$ and a $y \in M^{p_m}[g_m]$
  such that $y \notin \dot W_{\mathrm{I}}^{p_m}[g_m]$ and
  $y \notin \dot W_{\mathrm{II}}^{p_m}[g_m]$. This means that
  \[ M^{p_m}[g_m] \not\models \phi_{\mathrm{I}}^{p_m^\frown
      y}[\dot W_{\mathrm{I}}^{p_m^\frown y}] \text{ and }
    M^{p_m}[g_m] \not\models \phi_{\mathrm{II}}^{p_m^\frown
      y}[\dot W_{\mathrm{II}}^{p_m^\frown y}]. \]
      Let $p_{m+1} = p_m^\frown y$,
  $M^{p_{m+1}} = M^{p_m}[g_m]$. Here $M^{p_m}[g_m]$ can be rearranged as a
  $y_{m+1}$-premouse for some real $y_{m+1}$ coding $y_{m}$ and $g_m$
  as $g_m$ collapses a cutpoint of the $y_m$-premouse $M^{p_m}$. We will always
  consider $M^{p_{m+1}}$ as such a $y_{m+1}$-premouse.


  Suppose that $\gr(G_{p_m})$ is a limit ordinal and the rules of $G$
  dictate that after $p_m$, it is Player I's turn. By inductive
  hypothesis we have
  $M^{p_m}\not\models \phi^{p_m}_{\mathrm{I}}[\dot
  W_{\mathrm{I}}^{p_m}]$, so there is no active initial segment
  $M^{p_m}|\nu$, $\nu \in \Ord$, of $M^{p_m}$ of class
  $S_{\gr(G_{p_m^\frown k})}$ which is minimal and satisfies
  $\phi^{p_m^\frown k}_{\mathrm{I}}[\dot
  W_{\mathrm{I}}^{p_m^\frown k}]$, for any $k\in\omega$. Moreover,
  $M^{p_m}\not\models \phi^{p_m}_{\mathrm{II}}[\dot
  W_{\mathrm{II}}^{p_m}]$. Hence, there is some $k\in\omega$ and an
  active initial segment of $M^{p_m}$ of class
  $S_{\gr(G_{p_m^\frown k})}$ which is minimal and does not satisfy
  $\phi^{p_m^\frown k}_{\mathrm{II}}[\dot
  W_{\mathrm{II}}^{p_m^\frown k}]$, where $\delta^*$ is defined
  analogous as above. Fix such a $k$ and such an active initial
  segment $M^{p_m}|\eta$ of $M^{p_m}$ and let $M^{p_{m+1}}$ be the
  proper class model resulting from iterating the active extender of
  $M^{p_m}|\eta$ out of the universe. Set $p_{m+1} = p_m^\frown k$ and
  $y_{m+1} = y_m$. Then,
  \[ M^{p_{m+1}} \not\models
    \phi^{p_{m+1}}_{\mathrm{I}}[\dot
    W_{\mathrm{I}}^{p_{m+1}}] \text{ and } M^{p_{m+1}} \not\models
    \phi^{p_{m+1}}_{\mathrm{II}}[\dot
    W_{\mathrm{II}}^{p_{m+1}}]. \] The case that $\gr(G_{p_m})$ is a
  limit ordinal and the rules of $G$ dictate that it is Player II's turn after $p_m$ is similar.

  Finally, suppose that $\gr(G_{p_m}) = 2$. By Neeman's theorem (Theorem
  \ref{thm:Neeman}), there is some $M^{p_m}$-generic
  $g_m\subset\Col(\omega,\delta_{p_m})$ and some $y \in M_m[g_m]$ such
  that $y \notin \dot W^{p_m}_{\mathrm{I}}[g_m]$ and
  $y \notin \dot W^{p_m}_{\mathrm{II}}[g_m]$. This means that
\begin{enumerate}
\item $M^{p_m}[g_m]\models$ ``Player I does not have a winning strategy in
  $G_{p_m^\frown y}$'', and
\item $M^{p_m}[g_m]\models$ ``Player II does not have a winning
  strategy in $G_{p_m^\frown y}$''.
\end{enumerate}
However, $\gr(G_{p_m^\frown y}) = 1$, so $G_{p_m^\frown y}$ is a clopen
game of length $\omega$. This is a contradiction, as $M^{p_m}[g_m]$ is
a model of $\ZFC$, so it certainly satisfies clopen determinacy. This
proves \eqref{proofsd3}.

We now prove \eqref{proofsd1}; the proof of \eqref{proofsd2} is
similar.

Let $M$ be as in the statement of the claim and suppose
$M \models \phi_{\mathrm{I}}^{\emptyset}[\dot
W_{\mathrm{I}}^{\emptyset}]$. We will describe a winning strategy $\sigma$
for Player I in $G$ (in $V$) as a concatenation of strategies
$\sigma_m$ for different rounds of $G$. Playing against arbitrary
moves of Player II, we will for $m \geq 1$ inductively construct plays
$p_m$ for initial segments of $G$ according to $\sigma$ together with
\begin{itemize}
\item reals $y_m$ such that $y_{m+1} \geq_T y_m$ and $y_0 = r$,
\item $y_m$-premice $M^{p_m}$ which are proper class models of $\ZFC$
  of class $S_{\gr(G_{p_m})}$, none of whose initial segments are of class $S_{\gr(G_{p_m})}$, and
  such that $p_m \in M^{p_m}$,
\item premice $N^{p_m}$ together with iteration embeddings
  $j_m \colon M^{p_m} \rightarrow N^{p_m}$,
\item if $M^{p_m}$ has a Woodin cardinal and $\delta_{p_m}$ is the
  least Woodin cardinal in $M^{p_m}$,
  $\Col(\omega,j(\delta_{p_{m}}))$-generics $g_{p_m}$ over $N^{p_m}$.
\end{itemize}
In case $M^{p_m}$ does not have a Woodin cardinal, we will let $g_{p_m}$ be
undefined and let $M^{p_{m+1}} = N^{p_m}$. In the other case, we let
$M^{p_{m+1}} = N^{p_m}[g_{p_m}]$. Finally, we will stop the construction
after finitely many steps when $\gr(G_{p_m}) = 1$.

Let $p_0$ be the empty play and $M^{p_0} = M$. We will inductively
argue that
\[ M^{p_m} \models \phi_{\mathrm{I}}^{p_m}[\dot
  W_{\mathrm{I}}^{p_m}(M^{p_m})], \] where
$\dot W_{\mathrm{I}}^{p_m}(M^{p_m})$ is the
$\Col(\omega, \delta_{p_m})$-name or set of natural numbers in
$M^{p_m}$ defined above.



Assume inductively that $p_n$ and $M^{p_n}$ with
$M^{p_n} \models \phi_{\mathrm{I}}^{p_n}[\dot
W_{\mathrm{I}}^{p_n}(M^{p_n})]$ are already constructed
for all $n \leq m$ and that $\gr(G_{p_m}) \geq 2$. To construct
$p_{m+1}$ and $M^{p_{m+1}}$ we distinguish the following cases.

Assume first that $\gr(G_{p_m}) = \gamma+1$ for some $\gamma\geq
2$. Since
\[ M^{p_m} \models \phi_{\mathrm{I}}^{p_m}[\dot
  W_{\mathrm{I}}^{p_m}(M^{p_m})], \] it follows from
Neeman's theorem (Theorem \ref{thm:Neeman}) that there is a strategy
$\sigma_m$ for Player I in a game of length $\omega$ such that
whenever $x$ is a play by $\sigma_m$, there is an iterate $N^{p_m}$ of
$M^{p_m}$, an elementary embedding
$j \colon M^{p_m} \rightarrow N^{p_m}$ and an $N^{p_m}$-generic
$h\subset\Col(\omega,j(\delta_{p_m}))$ such that $x\in N^{p_m}[h]$ and
$x \in j(\dot W^{p_m}_{\mathrm{I}}(M^{p_m})[h]$. Hence,
by Claim \ref{cl:nameselememb},
$x \in \dot W^{p_m}_{\mathrm{I}}(N^{p_m})[h]$. By
definition,
\[ N^{p_m}[h] \models \phi_{\mathrm{I}}^{p_m^\frown x}[\dot
  W^{p_m^\frown x}_{\mathrm{I}}(N^{p_m}[h])]. \] 
  Let $p_{m+1} = p_m^\frown x$, $M^{p_{m+1}} = N^{p_m}[h]$, and $g_{p_{m}} = h$. Note that
$p_{m+1} \in M^{p_{m+1}}$ and as before $N^{p_m}[h]$ can be rearranged
as a $y_{m+1}$-premouse for some real $y_{m+1}$ coding $y_{m}$ and $h$. Again, we will
always consider $M^{p_{m+1}}$ as such a $y_{m+1}$-premouse.


Now suppose that $\gr(G_{p_m}) = \lambda$ is a limit ordinal and the rules of $G$ dictate that, after $p_m$, it is Player I's turn. By
assumption,
$M^{p_m} \models \phi_{\mathrm{I}}^{p_m}[\dot
W_{\mathrm{I}}^{p_m}(M^{p_m})]$. So 
\[ M^{p_m} \models \dot W_{\mathrm{I}}^{p_m}(M^{p_m})
  \neq \emptyset. \] Hence, there is a natural number $k$ and an
ordinal $\eta$ such that $M^{p_m}|\eta$ is an active initial segment
of $M^{p_m}$ of class $S_{\gr(G_{p_m^\frown k})}$ and is minimal, and, if we let $N^{p_m}$ be the result of
iterating the active extender of $M^{p_m}|\eta$ out of the universe,
then
$N^{p_m} \models \phi_{\mathrm{I}}^{p_m^\frown k}[\dot
W_{\mathrm{I}}^{p_m^\frown k}(N^{p_m})].$ 
Let $p_{m+1} = p_m^\frown k$,
$y_{m+1} = y_m$, and $M^{p_{m+1}} = N^{p_m}$. Moreover, let $\sigma_m$ be the strategy which
tells Player I to play $k$.

For the other limit case suppose that $\gr(G_{p_m}) = \lambda$ is a
limit ordinal and the rules of $G$ dictate that, after $p_m$, it is
Player II's turn. In this case we can let $\sigma_m = \emptyset$ as
Player I is not playing, i.e., we only have to react to what Player II
is playing in the next round. Suppose that Player II plays some
natural number $k$ and let $p_{m+1} = p_m^\frown k$. Since
$M^{p_m} \models \phi_{\mathrm{I}}^{p_m}[\dot
W_{\mathrm{I}}^{p_m}(M^{p_m})]$, we have
\[ M^{p_m} \models \dot W_{\mathrm{I}}^{p_m}(M^{p_m}) =
  \omega. \] In particular,
$k \in \dot W_{\mathrm{I}}^{p_m}(M^{p_m})$. Thus there is
an ordinal $\eta$ such that $M^{p_m}|\eta$ is an active initial
segment of $M^{p_m}$, minimal of class $S_{\gr(G_{p_m^\frown k})}$, and if we let $N^{p_m}$ denote the
result of iterating the active extender of $M^{p_m}|\eta$ out of the
universe, then
$N^{p_m} \models \phi_{\mathrm{I}}^{p_m^\frown k}[\dot
W_{\mathrm{I}}^{p_m^\frown k}(N^{p_m})].$ Let
$M^{p_{m+1}} = N^{p_m}$, and $y_{m+1} = y_m$.

Finally, assume that $\gr(G_{p_m}) = 2$. Since
$M^{p_m}\models\phi^{p_m}_{\mathrm{I}}[\dot
W^{p_m}_{\mathrm{I}}(M^{p_m} )]$, by Neeman's theorem (Theorem \ref{thm:Neeman}), there is a strategy
$\sigma_m$ for Player I in a game of length $\omega$ such that
whenever $x$ is a play by $\sigma_m$, there is an iterate $N^{p_m}$ of
$M^{p_m}$, an elementary embedding
$j \colon M^{p_m} \rightarrow N^{p_m}$ and an $N^{p_m}$-generic
$h\subset\Col(\omega,j(\delta_{p_m}))$ such that $x\in N[h]$ and
$x \in j(\dot W^{p_m}_{\mathrm{I}}(M^{p_m}))[h]$. Then,
by Claim \ref{cl:nameselememb},
$x \in \dot W^{p_m}_{\mathrm{I}}(N^{p_m}))[h]$. Therefore,
\[ N^{p_m}[h] \models \text{ Player I has a winning strategy in }
  G_{p_m^\frown x}. \] We let $p_{m+1} = p_m^\frown x$ and stop the
construction. Since $\gr(G_{p_m^\frown x}) = 1$, $G_{p_m^\frown x}$ is
a clopen game of length $\omega$. As $N^{p_m}[h]$ is a proper class
model of $\ZFC$, we can use absoluteness of winning strategies for
clopen games of length $\omega$ to obtain that Player I has a winning
strategy in $G_{p_m^\frown x}$ in $V$. Let $\sigma_{m+1}$ be a
strategy for Player I witnessing this.

This process describes a winning strategy $\sigma$ for Player I in
$G$ by concatenating the strategies $\sigma_i$ for $1\leq i \leq m+1$, as desired.
\end{proof}

This finishes the proof of Theorem \ref{prop:simpleclopen}.
\end{proof}

\section{Further applications}\label{sec:concl}

In this section, we present some additional applications of the proof of 
Theorem \ref{prop:simpleclopen}. Since the proofs are similar, we simply sketch the differences.

\subsection{Longer games}
We begin by noting that the results presented so far generalize to longer games. In particular:

\begin{theorem}\label{TheoremLongSPD}
  Let $\theta$ be a countable ordinal. Suppose that for each
  $\alpha<\omega_1$ and each $y\in\mathbb{R}$ there is some $x\geq_T y$ and an $x$-premouse $M$ of class $S_\alpha$ above
  some $\lambda$ below which there are $\theta$ Woodin cardinals in
  $M$. Then $\sigma$-projective games of length
  $\omega\cdot \theta$ are determined.
\end{theorem}

The theorem is a consequence of a more general result akin to Theorem \ref{prop:simpleclopen}, namely, the determinacy of simple $\sigma$-projective games of length $\omega\cdot(\theta+\omega)$, in the following sense:
\begin{definition}\label{DefSimpleGeneral}
  Let $\Gamma$ be a collection of subsets of $\omega^{\omega\cdot\theta + \omega^2}$
  (with each $A \in \Gamma$ identified with a subset of $\omega^{\omega\cdot\theta + \omega\cdot n}$ for some
  $n \in \omega$ as in Definition \ref{DefSimple}). A game of length $\omega\cdot\theta + \omega^2$ is
  $\Gamma$-\emph{simple} if it is obtained as follows:
\begin{enumerate}
\item \label{DefSimpleGeneral1} For every $n \in \omega$, games in $\Gamma$ of fixed length
  $\omega\cdot\theta + \omega \cdot n$ are $\Gamma$-simple.
\item  \label{DefSimpleGeneral2} Let $n\in\omega$ and for each
  $i \in \omega$ let $G_i$ be a $\Gamma$-simple game. Then the game
  $G$ obtained as follows is $\Gamma$-simple: Players I and II take
  turns playing natural numbers for $\omega\cdot\theta + \omega \cdot n$ moves, i.e., $\theta + n$
  rounds in games of length $\omega$. Afterwards, Player I plays some
  $i\in\omega$. Players I and II continue playing according to the
  rules of $G_i$ (keeping the first $\omega\cdot\theta + \omega \cdot n$ natural numbers
  they have already played).
\item \label{DefSimpleGeneral3} Let $n\in\omega$ and for each
  $i \in \omega$ let $G_i$ be a $\Gamma$-simple game. Then the game
  $G$ obtained as follows is $\Gamma$-simple: Players I and II take
  turns playing natural numbers for $\omega\cdot\theta + \omega \cdot n$ moves, i.e., $\theta + n$
  rounds in games of length $\omega$. Afterwards, Player II plays some
  $i\in\omega$. Players I and II continue playing according to the
  rules of $G_i$ (keeping the first $\omega\cdot\theta + \omega \cdot n$ natural numbers
  they have already played).
\end{enumerate}
\end{definition}

The notion of \emph{game rank} in this context is defined as before:
if $G$ is a game of fixed length $\omega\cdot\theta+\omega\cdot n$, then
$\gr(G) = n$. If $G$ is obtained from games $G_0, G_1, \hdots$, and
from an ordinal $\omega \cdot \theta + \omega\cdot n$ as in Definition \ref{DefSimpleGeneral}, we
let
\[ \gr(G) = \sup\{\gr(G_i)+\omega:i\in\omega\} + n. \] 

\begin{theorem}
 Let $\theta$ be a countable ordinal. Suppose that for each
  $\alpha<\omega_1$ and each $y\in\mathbb{R}$ there is some $x\geq_T y$ and an $x$-premouse $M$ of class $S_\alpha$ above
  some $\lambda$ below which there are $\theta$ Woodin cardinals in
  $M$. Then, simple $\sigma$-projective games of length
  $\omega\cdot(\theta+\omega)$ are determined.
\end{theorem}

\begin{proof}[Proof Sketch]
The theorem is proved like Theorem \ref{prop:simpleclopen}: first, by arguing as in Proposition \ref{PropositionSC2}, one sees that is suffices to prove determinacy for simple clopen games of length $\omega\cdot(\theta+\omega)$. Let $G$ be a game of successor rank $\alpha$ definable from a parameter $x$. 
Given $y\geq_T x$ and an active $y$-premouse $M$ of class $S_\alpha$, one defines sets $\dot W_{\mathrm{I}}^p(M)$ and $\dot W_{\mathrm{II}}^p(M)$ and formulae 
$\phi^p_{\mathrm{I}}$ and $\phi^p_{\mathrm{II}}$ 
by induction on $\gr(G_p)$ as in the proof of Theorem \ref{prop:simpleclopen}, provided $\gr(G_p) < \gr(G)$. 

The difference is as follows: in the proof of Theorem \ref{prop:simpleclopen}, $\gr(G) = \gr(G_p)$ occurs when $p = \emptyset$; here, it happens when $p$ is a $\theta$-sequence of reals. Instead of defining $\dot W^p_{\mathrm{I}}$ in this case, we define only $\dot W^\emptyset_{\mathrm{I}}$ and $\dot W^\emptyset_{\mathrm{II}}$. These are names for sets of $\theta+1$-sequences of reals and are defined as in Case 2 in the proof of Theorem \ref{prop:simpleclopen}; namely, 
\[\dot W^\emptyset_{\mathrm{I}} = \big\{(\dot p, q)| q\Vdash^M_{\Col(\omega,\delta)}\phi^{\dot p}_{\mathrm{I}}[\dot W^{\dot p}_{\mathrm{I}}]\big\},\]
where $\dot p$ is a name for a $\theta + 1$-sequence of reals, $\delta$ is the least Woodin cardinal of $M$ above $\lambda$, and $\delta^*$ is the second Woodin cardinal of $M$ above $\lambda$, if it exists, and $\omega$, otherwise.
(Note that $\delta$ exists since, by assumption, $\alpha$ is a successor ordinal). $\dot W^\emptyset_{\mathrm{II}}$ is defined similarly. 
Once these names have been defined, one applies Theorem 2A.2 of \cite{Ne04} (the
general version of Theorem \ref{thm:Neeman}) so that (using the fact that $M$ has $\theta$ Woodin cardinals below $\lambda$) one obtains formulae $\phi_{\mathrm{I}}$ and
  $\phi_{\mathrm{II}}$ with parameters $\dot W^\emptyset_{\mathrm{I}}$ and $\dot W^\emptyset_{\mathrm{II}}$, such that one of the following holds:
\begin{enumerate}
\item If $M\models\phi_{\mathrm{I}}[\dot W^\emptyset_{\mathrm{I}}]$, there is a strategy
  $\sigma$ for Player I in a game of length $\omega\cdot\theta + \omega$ such that
  whenever $\vec x$ is a play by $\sigma$, there is a non-dropping iterate $N$ of $M$ with an embedding $j$
  and an $N$-generic $g\subset\Col(\omega,j(\delta))$ such that
  $\vec x\in N[g]$ and $\vec x \in j(\dot W^\emptyset_{\mathrm{I}})[g]$.
\item If $M\models\phi_{\mathrm{II}}[\dot W^\emptyset_{\mathrm{II}}]$, there is a strategy
  $\tau$ for Player II in a game of length $\theta\cdot\omega+\omega$ such that whenever
  $\vec x$ is a play by $\tau$, there is a non-dropping iterate $N$ of $M$ with an embedding $j$ and an
  $N$-generic $g\subset\Col(\omega,j(\delta))$ such that $\vec x\in N[g]$ and
  $\vec x \in j(\dot W^\emptyset_{\mathrm{II}})[g]$.
\item Otherwise, there is an $M$-generic $g\subset\Col(\omega,\delta)$
  and an $\vec x\in M[g]$ such that $\vec x\not\in \dot W^\emptyset_{\mathrm{I}}[g]$ and
  $\vec x\not\in \dot W^\emptyset_{\mathrm{II}}[g]$.
\end{enumerate}

An argument as in the proof of Theorem \ref{prop:simpleclopen} yields the following claim:
\begin{claim}
  Let $M$ be an $x$-premouse and $\lambda \in M$ be an ordinal such that $M$ has $\theta$ Woodin cardinals below $\lambda$. Suppose that $M$ is of class $S_{\alpha}$ above $\lambda$ and no proper initial segment of $M$ is of class $S_{\alpha}$ above $\lambda$, where $\alpha$ is a successor ordinal. Let $\delta$ denote the least Woodin cardinal of $M$ above $\lambda$. Then
  \begin{enumerate}
  \item  If
    $M \models \phi^\emptyset_{\mathrm{I}}[\dot
    W_{\mathrm{I}}^\emptyset]$, then Player I has a winning strategy
    in $G$.
  \item  If
    $M \models \phi^\emptyset_{\mathrm{II}}[\dot
    W_{\mathrm{II}}^\emptyset]$, then Player II has a winning strategy
    in $G$.
  \item 
    $M\models\phi^\emptyset_{\mathrm{I}}[\dot
    W_{\mathrm{I}}^\emptyset] \vee
    \phi^\emptyset_{\mathrm{II}}[ \dot
    W_{\mathrm{II}}^\emptyset]$.
  \end{enumerate}
\end{claim}
The theorem is now immediate from the claim.
\end{proof}

It seems very likely that the hypotheses of Theorem \ref{TheoremLongSPD} are optimal. However, the proof in \cite{Ag18} does not seem to adapt easily to show this.

\subsection{$\sigma$-algebras} 
The proof of Theorem \ref{prop:simpleclopen} adapts to prove the determinacy of various $\sigma$-algebras. As an example, we
consider the smallest $\sigma$-algebra containing all projective sets. We show that a sufficient condition for its determinacy is the existence of $x$-premice of class $S_{\omega+1}$ for every $x\in\mathbb{R}$. 

\begin{theorem}
Suppose that for each $x\in\omega^\omega$ there is an $x$-premouse of class $S_{\omega+1}$. Then, every set in the smallest $\sigma$-algebra on $\omega^\omega$ containing the projective sets is determined.
\end{theorem}
\begin{proof}[Proof Sketch]
Let $A$ belong to the $\sigma$-algebra in the statement. Let
\[\{\SIGMA^0_\alpha(\PI^1_\omega):\alpha<\omega_1\}\]
denote the Borel hierarchy built starting from sets which are countable intersections of projective sets, i.e., $\SIGMA^0_0(\PI^1_\omega) = \PI^1_\omega$ consists of all countable intersections of projective sets and $\SIGMA^0_\alpha(\PI^1_\omega)$ consists of all countable unions of sets each of which is the complement of a set in $\SIGMA^0_\beta(\PI^1_\omega)$ for some $\beta<\alpha$. Standard arguments show that these pointclasses constitute a hierarchy of sets in the smallest $\sigma$-algebra containing the projective sets. It follows that there is $\alpha<\omega_1$ and $x\in\omega^\omega$ such that $A \in \Sigma^0_\alpha(\Pi^1_\omega)(x)$, i.e., such that $A$ belongs to $\SIGMA^0_\alpha(\PI^1_\omega)$ and that $A$ has a $\sigma$-projective code (in the sense of Section \ref{sec:sigmaproj}) which is recursive in $x$.

Let $[A]$ be a $\sigma$-projective code for $A$ which is recursive in $x$ and in which no complements appear. To show that $A$ is determined, it suffices to show that the decoding game for $[A]$ is determined. Let us denote this game by $G$. It is a simple clopen game; it is not quite of rank $\omega+1$, so determinacy does not follow immediately from Theorem \ref{prop:simpleclopen}, but it follows from the proof:

Given a partial play $p$ of the game and a model $N$, we define sets $\dot W_{\mathrm{I}}^p(N)$ and $\dot W_{\mathrm{II}}^p(N)$, formulae $\phi_{\mathrm{I}}[\dot W_{\mathrm{I}}^p(N)]$ and $\phi_{\mathrm{II}}[\dot W_{\mathrm{II}}^p(N)]$, and names $\dot B^{\dot p,I}_n$ and $\dot B^{\dot p,II}_n$ as in the proof of Theorem \ref{prop:simpleclopen}. The cases where $\gr(G_p)\leq\omega$ are exactly as in the proof of Theorem \ref{prop:simpleclopen}. The remaining cases are very slightly different---for this, let us define the notion of \emph{subrank}. Recall that $G$ has the following rules:
\begin{enumerate}
\item Players I and II begin by alternating $\omega$ many rounds to play a real number $y\in\omega^\omega$.
\item Afterwards, they alternate a finite amount of turns as follows: letting $[A_0] = [A]$ and supposing $[A_n]$ has been defined, $[A_n]$ is a $\sigma$-projective code for a union or an intersection of sets, say $\{B_i:i\in\omega\}$, with each $B_i$ of smaller (Borel) rank. By the rules of the decoding game, one of the players needs to play a natural number $i$, thus selecting a $\sigma$-projective code $[B_i]$ for $B_i$; we let $[A_{n+1}] = [B_i]$.
\item Eventually, a stage $n^*$ is reached in which $[A_{n^*}]$ is a code for a projective set; at this point the players continue with the rules of the decoding game for $[A_{n^*}]$.
\end{enumerate}
Given a play $p$ of $G$ in which a player needs to play a natural number $i$ and $[A_n]$ has been defined as above, we say that the \emph{game subrank} of $G_p$ is the least $\gamma$ such that $A_n \in \Sigma^0_\gamma(\Pi^1_\omega)(x)$ or $A_n \in \Pi^0_\gamma(\Sigma^1_\omega)(x)$.

We now continue the definition of the sets $\dot W_{\mathrm{I}}^p(N)$ and $\dot W_{\mathrm{II}}^p(N)$, formulae $\phi_{\mathrm{I}}[\dot W_{\mathrm{I}}^p(N)]$ and $\phi_{\mathrm{II}}[\dot W_{\mathrm{II}}^p(N)]$, and names $\dot B^{\dot p,I}_n$ and $\dot B^{\dot p,II}_n$. Assume $p$ is such that $\omega<\gr(G_p)$, so that the sets have not been defined already; we proceed by induction on the subrank of $G_p$:
\begin{case}
  The subrank of $G_p$ is defined and nonzero and the rules of $G$ dictate
  that, after $p$, it is Player I's turn.
\end{case}

Let $y \geq_T x$ and let $M$ be any $y$-premouse which is a model of
$\ZFC$ and of class $S_\omega$. Let $p \in M$. Then let
$\dot W^p_{\mathrm{I}} = \dot W^p_{\mathrm{I}}(M)$ be the set
of all $k\in\omega$ such that 
\[M \models \phi^{p^\frown k}_{\mathrm{I}}[\dot
  W_{\mathrm{I}}^{p^\frown k}(M)].\] 
  This makes sense because for all
$k\in\omega$, the subrank of $G_{p^\frown k}$ is smaller than the subrank of $G_p$, so the set
$\dot W_{\mathrm{I}}^{p^\frown k}(M)$ has been defined.  Moreover, $\dot W^p_{\mathrm{I}}$ belongs to $M$,
since $[A]$ is recursive in $x$ and the formulae
$\phi^{p^\frown k}_I(\dot W_{\mathrm{I}}^{p^\frown
  k}[M])$, as well as the sets
$\dot W_{\mathrm{I}}^{p^\frown k}(M)$, are
definable uniformly in $p$, as can be shown inductively by following this
construction and the proof of Theorem \ref{thm:Neeman} (cf. the proof
of \cite[Theorem 1E.1]{Ne04}).

We let $\phi^p_{\mathrm{I}}[\dot W_{\mathrm{I}}^p(M)]$
be the formula asserting that $\dot W^p_{\mathrm{I}}$ is
non-empty. Similarly, let
$\dot W^p_{\mathrm{II}} = \dot W^p_{\mathrm{II}}(M)$ be the set
of all $k\in\omega$ such that
\[M \models \phi^{p^\frown k}_{\mathrm{II}}[\dot
  W_{\mathrm{II}}^{p^\frown k}(M)].\]
   We let
$\phi^p_{\mathrm{II}}[\dot W_{\mathrm{II}}^p(M)]$ be the
formula asserting that $\dot W^p_{\mathrm{II}}$ is equal to
$\omega$. As before, the set $\dot W^p_{\mathrm{II}}$ belongs to $M$.

The good $\mathbb{P}_n$-names $\dot B^{\dot p,I}_n$ and $\dot B^{\dot p,II}_n$ are defined as in the other cases.

\begin{case}
  The subrank of $G_p$ is defined and nonzero and the rules of $G$ dictate
  that, after $p$, it is Player II's turn.
\end{case}
This case is similar to the preceding one.

\begin{case}
  $p$ is the empty play.
\end{case}
Let $M$ be any $x$-premouse which is a model of
$\ZFC$ and of class $S_{\omega+1}$ and let $\delta$ be the smallest Woodin cardinal of $M$.
In this case, $\dot W^\emptyset_{\mathrm{I}}$ is the canonical $\Col(\omega,\delta)$-name for all reals $p$ such that $\phi^{p}_I(\dot W_{\mathrm{I}}^{p}[\dot B])$ holds, where $\dot B \in M$ is a good $\Col(\omega,\delta)$-name with respect to
  $\dot p$, so that whenever $G$ is
  $\Col(\omega, \delta)$-generic over $M$,
  $\dot B[G] = \dot W^{p}_{\mathrm{I}}(M[G])$,
  where $p = \dot p[G]$.

The name $\dot W^\emptyset_{\mathrm{II}}$ is defined analogously. The formulae $\phi^\emptyset_{\mathrm{I}}(\dot W^\emptyset_{\mathrm{I}})$ and $\phi^\emptyset_{\mathrm{II}}(\dot W^\emptyset_{\mathrm{II}})$ are obtained by applying Neeman's theorem to $\dot W^\emptyset_{\mathrm{I}}$ and $\dot W^\emptyset_{\mathrm{II}}$. \\

This completes the definition of the names and formulae. After this, an argument as in the proof of Theorem \ref{prop:simpleclopen} yields the following claim:
\begin{claim}
  Let $M$ be an $x$-premouse which is of class $S_{\omega+1}$ but has no proper initial segment of class $S_{\omega+1}$. Let $\delta$ denote the least Woodin cardinal
  in $M$. Then
  \begin{enumerate}
  \item  If
    $M \models \phi^\emptyset_{\mathrm{I}}[\dot
    W_{\mathrm{I}}^\emptyset]$, then Player I has a winning strategy
    in $G$.
  \item  If
    $M \models \phi^\emptyset_{\mathrm{II}}[\dot
    W_{\mathrm{II}}^\emptyset]$, then Player II has a winning strategy
    in $G$.
  \item 
    $M\models\phi^\emptyset_{\mathrm{I}}[\dot
    W_{\mathrm{I}}^\emptyset] \vee
    \phi^\emptyset_{\mathrm{II}}[\dot
    W_{\mathrm{II}}^\emptyset]$.
  \end{enumerate}
\end{claim}
The theorem is now immediate from the claim.
\end{proof}

The following remains an interesting open problem:
\begin{question}
What is the consistency strength of determinacy for sets in the smallest $\sigma$-algebra containing the projective sets?
\end{question}

\section*{Acknowledgements}

The first-listed author was partially suported by FWO grant number 3E017319 and by FWF grant numbers I4427 and I4513-N.
The second-listed author, formerly known as Sandra Uhlenbrock, was partially supported by FWF grant number P
28157 and in addition gratefully acknowledges funding from L'OR\'{E}AL Austria, in collaboration with the Austrian UNESCO Commission and in cooperation with the Austrian Academy of Sciences - Fellowship \emph{Determinacy and Large Cardinals}. 
This project has received funding from the European Union's Horizon 2020 research and innovation programme under the Marie Sk{\l}odowska-Curie grant agreement No 794020 of the third-listed author (Project \emph{IMIC: Inner models and infinite computations}). The third-listed author was partially supported by FWF grant number I4039.

\bibliographystyle{alpha}
\bibliography{References}
  
\end{document}